\documentclass[twoside,12pt]{amsart}
\usepackage[T1]{fontenc}
\usepackage{amsfonts}
\usepackage{amssymb}
\usepackage{amsthm}
\usepackage{amsmath}
\usepackage{mathrsfs} 
\usepackage{color}
\usepackage{graphicx}
\usepackage{caption}
\usepackage{float}
\usepackage{url}
\usepackage[all]{xy}
\usepackage{lscape}
\theoremstyle{plain}
\usepackage{array}						
\usepackage{mathtools}
\usepackage[vcentermath]{youngtab}
\usepackage[margin=3.3cm]{geometry}

\newcommand{\fS}{\mathfrak{S}}

\newcommand{\fl}{\mathfrak{l}}

\newcommand{\fs}{\mathfrak{s}}

\newcommand{\sL}{\mathscr{L}}

\newcommand{\sQ}{\mathscr{Q}}

\newcommand{\sS}{\mathscr{S}}

\newcommand{\cA}{\mathcal{A}}

\newcommand{\cF}{\mathcal{F}}

\newcommand{\cH}{\mathcal{H}}

\newcommand{\cO}{\mathcal{O}}

\newcommand{\Z}{\mathbb{Z}}

\newcommand{\Q}{\mathbb{Q}}
\newcommand{\C}{\mathbb{C}}

\newcommand{\bc}{\mathbf{c}}

\newcommand{\br}{\mathbf{r}}
\newcommand{\bs}{\mathbf{s}}
\newcommand{\bt}{\mathbf{t}}

\newcommand{\al}{\alpha}
\newcommand{\be}{\beta}
\newcommand{\si}{\sigma}
\newcommand{\la}{\lambda}
\newcommand{\ga}{\gamma}

\newcommand{\ka}{\kappa}

\newcommand{\de}{\delta}

\newcommand{\La}{\Lambda}

\newcommand{\De}{\Delta}

\newcommand{\bla}{\boldsymbol{\la}}
\newcommand{\bnu}{\boldsymbol{\nu}}
\newcommand{\bmu}{\boldsymbol{\mu}}

\newcommand{\ta}{\tilde{a}}
\newcommand{\tb}{\tilde{b}}

\newcommand{\te}{\tilde{e}}
\newcommand{\tf}{\tilde{f}}

\newcommand{\ts}{\tilde{s}}

\newcommand{\tF}{\tilde{F}}

\newcommand{\dw}{\dot{w}}
\newcommand{\dA}{\dot{A}}
\newcommand{\dS}{\dot{S}}
\newcommand{\ds}{\dot{s}}
\newcommand{\dB}{\dot{B}}
\newcommand{\dL}{\dot{L}}
\newcommand{\dN}{\dot{N}}
\newcommand{\dM}{\dot{M}}
\newcommand{\dm}{\dot{m}}
\newcommand{\dT}{\dot{T}}
\newcommand{\dtau}{\dot{\tau}}
\newcommand{\dka}{\dot{\ka}}
\newcommand{\dvarphi}{\dot{\varphi}}
\newcommand{\dPsi}{\dot{\Psi}}

\newcommand{\dsS}{\dot{\sS}}

\newcommand{\dbs}{\dot{\bs}}

\newcommand{\dbt}{\dot{\bt}}
\newcommand{\dbla}{\dot{\bla}}

\newcommand{\dbnu}{\dot{\bnu}}

\newcommand{\tde}{\tilde{\dot{e}}}
\newcommand{\tdf}{\tilde{\dot{f}}}

\newcommand{\tdF}{\tilde{\dot{F}}}

\newcommand{\tbla}{\tilde{\bla}}

\newcommand{\tbs}{\tilde{\bs}}

\newcommand{\dtbs}{\dot{\tbs}}

\newcommand{\hw}{\widehat{w}}
\newcommand{\hv}{\widehat{v}}
\newcommand{\hP}{\widehat{P}}

\newcommand{\lra}{\longrightarrow}
\newcommand{\Ra}{\Rightarrow}

\newcommand{\eq}{\Leftrightarrow}

\newcommand{\Ue}{\mathcal{U}'_q (\widehat{\mathfrak{sl}_e})}
\newcommand{\Ul}{\mathcal{U}'_p (\widehat{\mathfrak{sl}_l})}
\newcommand{\Uinf}{\mathcal{U}_q ({\mathfrak{sl}_\infty})}

\newcommand{\sle}{\widehat{\mathfrak{sl}_e}}
\newcommand{\sll}{\widehat{\mathfrak{sl}_l}}

\newcommand{\dispst}{\displaystyle}
\newcommand{\mand}{\quad\text{and}\quad}
\newcommand{\Irr}{\text{\rm Irr}}

\newcommand{\id}{\text{\rm Id}}

\newcommand{\bemp}{\boldsymbol{\emptyset}}
\newcommand{\dbemp}{\dot{\bemp}}

\newcommand{\res}{\text{\rm res}}
\newcommand{\cont}{\text{\rm cont}}

\newtheorem{num}{Notation}[section]
\newtheorem{defi}[num]{Definition}
\newtheorem*{defi*}{Definition}

\newtheorem{thm}[num]{Theorem}
\newtheorem*{thm*}{Theorem}
\newtheorem{lem}[num]{Lemma}
\newtheorem*{lem*}{Lemma}
\newtheorem{prop}[num]{Proposition}
\newtheorem*{prop*}{Proposition}
\newtheorem{cor}[num]{Corollary}
\newtheorem*{cor*}{Corollary}

\newtheorem*{conj*}{Conjecture}
\newtheorem*{rem*}{Remark}

\theoremstyle{remark}
\newtheorem{exa}[num]{Example}
\newtheorem{rem}[num]{Remark}

\begin{document}
\title{Triple crystal action in Fock spaces}
\author{Thomas Gerber}


\address{Lehrstuhl D f\"ur Mathematik, RWTH Aachen University,
52062 Aachen, Germany}

\email{gerber@math.rwth-aachen.de}


\begin{abstract}
We make explicit a triple crystal structure on higher level Fock spaces,
by investigating at the combinatorial level the actions of two affine quantum groups and of a Heisenberg algebra.
To this end, we first determine a new indexation of the basis elements that makes the two quantum group crystals commute.
Then, we define a so-called Heisenberg crystal, commuting with the other two.
This gives new information about the representation theory of 
cyclotomic rational Cherednik algebras, relying on some recent results of Shan and Vasserot and of Losev.
In particular, we give an explicit labelling of their finite-dimensional simple modules.
\end{abstract}

\maketitle

\tableofcontents

\markright{TRIPLE CRYSTAL ACTION IN FOCK SPACES}

\section{Introduction}

Since Ariki's proof \cite{Ariki1996} of the LLT conjecture \cite{LLT1996},
it is understood that higher level Fock spaces representations of $\Ue$
play, via categorification, a very important role in understanding
some classical structures related to complex reflection groups.
More precisely, if $\cF_{\bs,e}$ is the level $l$ Fock space representation of $\Ue$ with multicharge $\bs$
and $V(\bs)$ the irreducible highest weight submodule of $\cF_{\bs,e}$ of weight $\La_{\bs}$
(determined by $\bs$),
then one can compute the decomposition numbers for the corresponding Ariki-Koike algebra
by specialising at $q=1$ Kashiwara's canonical basis of $V(\bs)$.

The Fock space itself is no longer irreducible, but one can however define a canonical basis for it,
which turns out to give, at $q=1$, the decomposition numbers of a corresponding $q$-Schur algebra, as
was proved by Varagnolo and Vasserot \cite{VaragnoloVasserot1999}, hence generalising Ariki's result.

The introduction of quiver Hecke algebras by Rouquier \cite{Rouquier2008a}
and by Khovanov and Lauda \cite{KhovanovLauda2008} has shed some new light about the role of the parameter $q$.
In fact, quiver Hecke algebras are graded, and graded versions of these 
results (which do not require to specialise $q$ at $1$) hold for these structures, 
see \cite{BrundanKleshchev2009}.

Moreover, Ariki's categorification theorem also permits to interpret the Kashiwara crystal
of $V(\bs)$ as the branching rule for the associated Ariki-Koike algebra \cite{Ariki2007}.
Shan has proved in \cite{Shan2011} that the crystal of the whole Fock space
is also categorified by a branching rule, but for another structure, namely a
corresponding cyclotomic rational Cherednik algebra.

Very recently, Dudas, Varagnolo and Vasserot \cite{DudasVaragnoloVasserot2015} have proved
a similar result suggested by Gerber, Hiss and Jacon \cite{GerberHissJacon2015}
in the context of finite unitary groups.
In this case, there is a notion of parabolic (or Harish-Chandra) induction for unipotent representations
which, provided one works with the appropriate Levi subgroups,
defines a branching graph which also categorifies the crystal of $\cF_{\bs,e}$.

\medskip

Therefore, the study of $\cF_{\bs,e}$, and in particular of
its crystal structure, which yields the theory of canonical bases in Kashiwara's approach \cite{Kashiwara1993},
is crucial for approaching fundamental problems in the representation theory of many
classical algebraic structures.

In Uglov's paper \cite{Uglov1999}, canonical bases of higher level Fock spaces have been
thoroughly studied, generalising Leclerc and Thibon's level one results in \cite{LeclercThibon1996} and
\cite{LeclercThibon2001}.
In his work, the Fock space is identified with a subspace of the level one Fock space $\La^s$.
This space $\La^s$ is essentially the direct sum of all Fock spaces $\cF_{\bs,e}$
over all $l$-charges $\bs$ whose components sum up to $s$.
Three algebras act on $\La^s$:
two quantum groups, namely $\Ue$ and $\Ul$, where $p=-1/q$, and a Heisenberg algebra $\cH$.
A fundamental result is that these three actions are pairwise commutative \cite[Proposition 4.6]{Uglov1999}, 
and that the Fock space can be decomposed in a very simple way:
it suffices to act on the empty multipartition by the three algebras for some
restricted values of the multicharge to reach any vector \cite[Theorem 4.8]{Uglov1999}.

This emphasizes the relevance of considering $\La^s$ not only as a $\Ue$-module, but also
as a $\Ul$-module and as a $\cH$-module.
This triple module structure is well-defined because one has a three natural ways to index
the elements of $\La^s$: either by partitions, by $l$-partitions, or by $e$-partitions.
The action of $\Ul$ is understood provided the correspondence between $l$-partitions and $e$-partitions
is known, which is the case (it is explicit and based on taking $e$-quotients and ``modified'' $l$-quotients).
This double quantum group structure is referred to as ``level-rank'' duality, and
has been investigated in particular in the works of Rouquier, Shan, Varagnolo and Vasserot \cite{RSVV2016} and Webster \cite{Webster2013a}
in the context of rational Cherednik algebras.

The action of $\cH$ has been less studied, but has recently been proved to have important applications
when it comes to the representation theory of rational Cherednik algebras.
More precisely, Shan and Vasserot \cite{ShanVasserot2012} have categorified the Heisenberg action on the Fock space, and used it
to characterise finite-dimensional simple modules for the corresponding cyclotomic Cherednik algebra.
In a recent preprint \cite{Losev2015}, Losev has given a combinatorial interpretation of this categorical action,
but without using Uglov's approach to the Fock space.
Finally, in the context of finite classical groups,
it has been shown that $\cH$ plays a role in the study of unipotent modular representations,
by relating the notion of weak cuspidality to the classical one, see \cite[Section 5]{DudasVaragnoloVasserot2016}.

However, Uglov's work says very little about crystals, and it is not clear 
how level-rank duality nor the action of $\cH$ is expressed at the crystal level.

\medskip

The aim of this paper two-fold.
First, complete Uglov's study of higher level Fock spaces at the crystal level.
This is achieved by explicitely determining a triple crystal structure which yields a nice
decomposition of the whole crystal in the spirit of \cite[Proposition 4.6]{Uglov1999} and
\cite[Theorem 4.8]{Uglov1999}.
This requires to make explicit the $\Ul$-crystal structure commuting with the $\Ue$-crystal,
and to define an appropriate notion of \textit{Heisenberg crystal} which shall commute
with both affine quantum group crystals.
Second, place it into the context of the representation theory of Cherednik algebras to deduce new results
using explicit combinatorics.
In order to do this, we must in particular prove a compatibily with a recent result of Losev \cite{Losev2015}.
Throughout the paper, we will give a significant amount of examples to illustrate the notions
and procedures that we introduce.

This article has the following structure.
In Section 2, we recall the important combinatorial notions used in all the paper,
in particular Uglov's algorithms which permit to juggle the different indexations of
the level one Fock space.
This section does not contain any new material, however we take some time to
reintroduce all notions carefully.

In Section 3, we recall the $\Ue$-module structure on higher level Fock space
which was explicited in \cite{JMMO1991}.
Note that this requires an order on $i$-boxes of multipartitions which gives the so-called
``Uglov'' realisation of the Fock space, which is not consensual in the literature.
We also introduce the conjugation procedure, which is capital,
and the correspondence (\ref{indexation}).
Essentially, it enables to put a new $\Ue$- and $\Ul$-structure on the Fock space,
which is the appropriate one for our purpose of studying crystals.

Section 4 gives a crystal version of level-rank duality.
The $\Ul$-crystal graph rule commuting
with the classic $\Ue$-crystal is explicited (Section \ref{comm_crystals}).
The crystal operators of $\Ul$ therefore give
$\Ue$-crystal isomorphisms in level $l$ Fock spaces,
adding to the list in \cite{Gerber2015}.
It is explicit and easy to describe on $l$-partitions.
The use of Correspondence (\ref{indexation}) is indispensable.

In Section 5, we study ``doubly highest weight vertices'',
that is to say, vertices that are simultaneously highest weight vertices
in the $\Ue$-crystal and in the $\Ul$-crystal.
We use a result of Jacon and Lecouvey \cite[Theorem 5.9]{JaconLecouvey2012}
to characterise these multipartitions, and then give some
essential properties.
The proofs there are quite technical and require 
a careful analysis of Correspondence (\ref{indexation}).

Section 6 is devoted to defining the Heisenberg crystal.
We start by introducing maps $\tb_{-\ka}$ and $\tb_\si$
which shift periods to the left and to the right respectively
in the abacus representation of a multipartition.
These are simultaneously $\Ue$- and $\Ul$-crystal isomorphisms.
The Heisenberg crystal is then defined as a graph where the arrows
are given by the action of ``Heisenberg operators'' which are refined versions of
$\tb_{-\ka}$ and $\tb_\si$.
We end with a decomposition theorem (Theorem \ref{thm_crys_decomp}) which is
an analogue of \cite[Theorem 4.8]{Uglov1999}.

Finally, Section 7 relates the various results of the previous sections
to the representation theory of rational Cherednik algebras.
We first give a general interpretation of the crystal level-rank duality.
Then, we show that Losev's independent results on the crystal version of the Heisenberg action
\cite{Losev2015} are compatible with those of Section 6, this is Theorem \ref{thmosev2}.
This enables us to give an explicit characterisation of the finite-dimensional simple
modules for the cyclotomic rational Cherednik algebras using the notion of FLOTW multipartitions
(Theorem \ref{thm_cher}).

\section{General combinatorics}\label{gen_comb}

\subsection{Charged multipartitions}\label{charged_mp}\

Let $l$ be a positive integer.
An \textit{$l$-charge} (or simply multicharge) is an $l$-tuple $\bs=(s_1,\dots,s_l)$ of integers.
An \textit{$l$-partition} (or simply multipartition) is an $l$-tuple $\bla=(\la^1,\dots,\la^l)$ of partitions.
One considers that a partition has an infinite number of size zero parts.
The set of partition will be denoted by $\Pi$ and the set of $l$-partitions by $\Pi_l$.
The rank of $\bla$ is the sum of the ranks of the partitions $\la^j$, $j=1,\dots,l$.
A charged $l$-partition is the data of an $l$-charge $\bs$ and an $l$-partition $\bla$, denoted $|\bla,\bs\rangle$.
It can be represented by an $l$-tuple of Young diagrams (corresponding to $\bla$),
whose boxes $(a,b,j)$ (where $a$ is the row of the box, $b$ is its column and $j$ its component)
are filled with the integers $b-a+s_j$.
This integer is called the \textit{content} of the box $(a,b,j)$.

\newcommand{\moinsun}{$-$1}
\newcommand{\moinsdeux}{$-$2}

\begin{exa}\label{exa_chargedmp}
Take $l=2$, $\bs=(-1,2)$ and $\bla=(2.1,1^2)$.
Then
$$
|\bla,\bs\rangle = (\; \young(\moinsun0,\moinsdeux)\; , \; \young(2,1) \;)
$$
\end{exa}

\subsection{Abaci}\label{abaci}\

Equivalently, one can use $\Z$-graded abaci to represent a charged multipartition, following \cite{JamesKerber1984}.
Given a charged $l$-partition $|\bla,\bs\rangle$, we can compute, for each $j=1\,\dots,l$,
the numbers $\be_k^j= \la_k^j + s_j -k +1$ for $k\geq0$ where the $\la_k^j$ ($k\geq 1$) are the parts of $\la^j$.
For each $j$, this give a set of $\be$-numbers for $\la_j$ in the sense of \cite{JamesKerber1984},
which is infinite by the convention that $\la^j$ has an infinite number of size zero parts.
Note that these $\be$-numbers are precisely the ``virtual'' contents
that appear just to the right of the border of $|\bla,\bs\rangle$ in its Young diagram representation.

Formally, we define the abacus $\cA(\bla,\bs)$ to be the subset of $\{1,\dots,l\}\times\Z$
defined by
$$\cA(\bla,\bs) = \left\{ (j,\be_k^j) \; , \; j\in\{1,\dots,l\}, k\geq0 \right\}.$$
We can depict $\cA(\bla,\bs)$ by drawing $l$ horizontal $\Z$-graded rows, 
numbered from bottom to top, and by drawing a bead on row $j$ at position $\be_k^j$ for all $j=1,\dots,l$
and for all $k\geq0$.

\begin{exa}
Take the same values as Example \ref{exa_chargedmp}.
Then the $\be$-numbers are given by
\begin{equation*}
\begin{array}{ccl}
j=2: & \hspace{2cm} & (\dots, -5, -4, -3, -2, -1, 0 , 2, 3)\\
j=1: & \hspace{2cm} & (\dots, -5, -4, -3, -1, 1) 
\end{array}
\end{equation*}
and we get the following abacus
\begin{figure}[H] 
\includegraphics{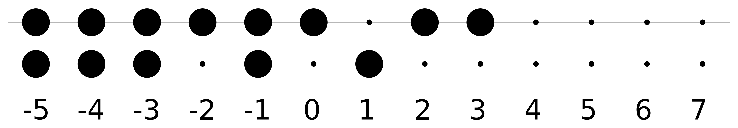}
\label{ab15}
\end{figure}
\end{exa}

From this abacus, one recovers 
the $l$-charge by shifting all beads to the left and by looking at the position of the rightmost bead on each row;
and the partition $\la^j$ (for all $j=1\dots,l$) by counting the number of empty spots to the left of each bead
on the $j$-th row.

\subsection{Uglov's algorithms}\label{uglovs_bij}\

In this section, we want $l\geq 2$ and we fix another integer $e\geq2$.
Following Uglov \cite{Uglov1999}, we explain a way to 
associate to a charged $l$-partition a charged $1$-partition, as well as a charged $e$-partition.

Consider the $l$-abacus representing a charged $l$-partition $|\bla,\bs\rangle$.
Divide it into rectangles $R_k$, with $k\in\Z$, of size $e\times l$ such that each rectangle contains the 
positions $(j,(k-1)e+1), (j,(k-1)e+2),\dots, (j,ke)$ for some $k\in\Z$ and for all $j\in\{1,\dots,l\}$.
Then for $(j,c)\in R_k$, set $\tau^{-1}(j,c) =(1, c-e(j-1)+elk)$.
Then one can show (see \cite{Uglov1999}) that $\tau^{-1}$ is a bijection between $\{1,\dots,l\}\times\Z$ and $1\times\Z$,
and we denote $\tau$ its inverse (whose formula can also be explicited).
In fact, $\tau^{-1}(\cA(\bla,\bs))$ is a $1$-abacus representing a charged partition, which we denote
$|\la,s\rangle$. It is easy to see that $s=\sum_{j=i}^ls_j$.

Starting from $|\la,s\rangle$, we can define a variant of $\tau$ which uses $e$.
For $(1,c)\in\cA(\la,s)$, set $\dtau(1,c)=((-c)\mod e+1 , \left\lceil \frac{c}{e} \right\rceil)$.
Then $\dtau$ is also bijection between $1\times\Z$ and $\{1,\dots,e\}\times\Z$.
In fact, $\dtau(\cA(\la,s))$ is an $e$-abacus representing a charged $e$-partition, which we denote
$|\dbla,\dbs\rangle$, and we also have $s=\sum_{i=1}^e s_i$ if $\dbs=(\ds_1,\dots,\ds_e)$.

\begin{exa}
We illustrate these procedures on Example \ref{exa_chargedmp}.
Take for instance $e=3$.
We see in Figure \ref{ab16} that $\tau$ ``stacks horizontally'' the elements of $\cA(\la,s)$ into
rectangles of height $l$ and width $e$, and that 
$\dtau$ ``stacks vertically'' the elements of $\cA(\la,s)$ into
rectangles of height $e$ and width $l$.
Notice also that the composition $\dtau \circ \tau^{-1}$ consists in
flipping each rectangle through the diagonal joining the top left corner to 
the bottom right corner, and then ```regluing'' the rectangles to get an abacus.
\begin{figure}[H] 
\includegraphics{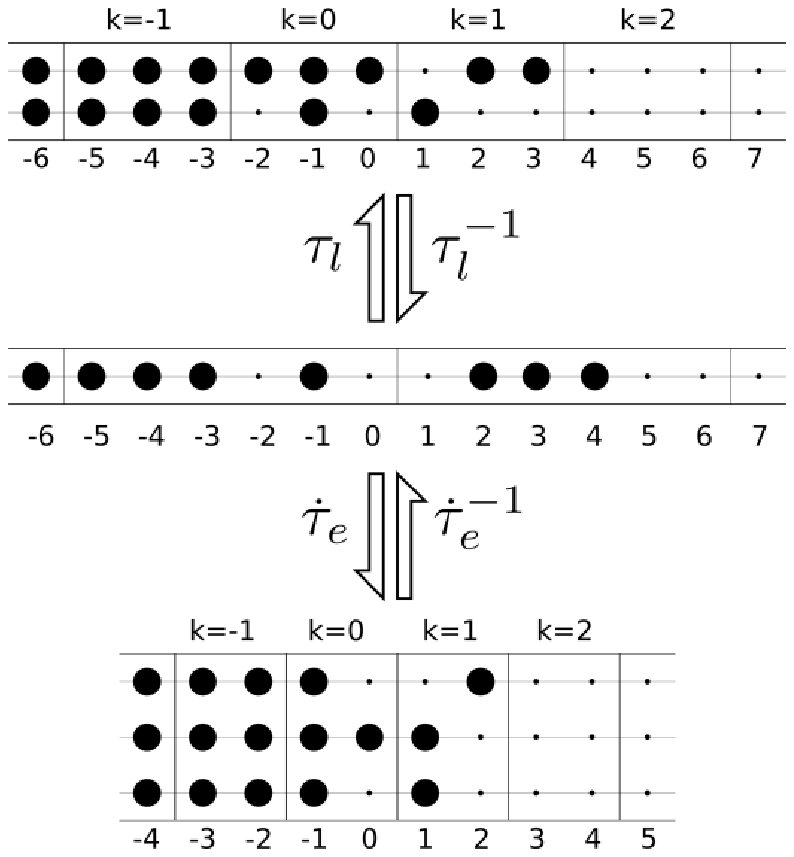}
\caption{The bijections $\tau$ and $\dtau$}
\label{ab16}
\end{figure}
\end{exa}

\begin{rem}\label{rem_ribbon}
Notice that shifting a bead one step to the left in $|\dbla,\dbs\rangle$ amounts to removing an $e$-ribbon in $|\la,s\rangle$.
In fact, the bijection $\dtau$ gives the \textit{$e$-quotient} (in the sense of \cite{JamesKerber1984})
of the partition $|\la,s\rangle$.
The \textit{$e$-core} of $\la$ is obtained after shifting all beads of $|\dbla,\dbs\rangle$ to the left
and computing the associated partition using $\dtau^{-1}$.
Note finally that this does not hold for $\tau$.
\end{rem}

\subsection{Addable/removable boxes, residues and order on boxes}\label{add_rem_boxes} \

We keep the notations of the previous section.
Recall that the content of a box $\ga=(a,b,j)$ of a multipartition $|\bla,\bs\rangle$
is the integer $\cont(\ga) = b-a+s_j$.
The \textit{residue} of $\ga$ is the integer $$\res(\ga) = \cont(\ga) \mod e.$$
For $i\in\{0,\dots,e-1\}$, $\ga$ is called an $i$-box if $\res(\ga)=i$.

A box $\ga$ is called \textit{removable} for $\bla$ if
$\ga$ is a box of $\bla$ and if $\bla \backslash \{\ga\}
$ is still a multipartition.
Similarly, it is called \textit{addable} if $\bla \cup \{\ga\}$ is still a multipartition.
In the abacus, this corresponds to a bead which can be shifted one step to the left (respectively to the right).
As seen in Remark \ref{rem_ribbon},
removing (respectively adding) a box in $|\dbla,\dbs\rangle$ 
corresponds to removing (respectively adding) an $e$-ribbon in $|\la,s\rangle=\dtau^{-1}(|\dbla,\dbs\rangle)$.

For a charged $l$-partition $|\bla,\bs\rangle$ and $i\in\{0,\dots,e-1\}$, 
there is a total order on the set of its removable and addable $i$-boxes defined by
\begin{equation}
\ga < \ga' \eq \left\{
\begin{array}{l}
\cont(\ga) < \cont(\ga') \text{\quad or} \\
\cont(\ga) = \cont(\ga')\mand j>j'
\end{array}
\right.
\end{equation}
where $\ga=(a,b,j)$ and $\ga'=(a',b',j')$.

For charged $l$-partitions $|\bla,\bs\rangle$
and $|\bmu,\bs\rangle$ such that $\bmu = \bla\cup\{\ga\}$
where $\ga$ is an addable $i$-box of $|\bla,\bs\rangle$, 
we define the quantities
\begin{equation}
\begin{array}{rl}
N_i(|\bla,\bs\rangle) = & \# \{ \text{addable $i$-boxes of $|\bla,\bs\rangle$} \} \\
& - \# \{ \text{removable $i$-boxes of $|\bla,\bs\rangle$} \}\\
\\
N_i^<(|\bla,\bs\rangle,|\bmu,\bs\rangle) = & \# \{ \text{addable $i$-boxes $\ga'$ of $|\bla,\bs\rangle$ such that $\ga'<\ga$} \} \\
& - \# \{ \text{removable $i$-boxes $\ga'$ of $|\bmu,\bs\rangle$ such that $\ga'<\ga$} \}\\
\\
N_i^>(|\bla,\bs\rangle,|\bmu,\bs\rangle) = & \# \{ \text{addable $i$-boxes $\ga'$ of $|\bla,\bs\rangle$ such that $\ga'>\ga$} \} \\
& - \# \{ \text{removable $i$-boxes $\ga'$ of $|\bmu,\bs\rangle$ such that $\ga'>\ga$} \}.\\
\end{array}
\end{equation}

\section{Module structures on the Fock space}\label{fock}

\subsection{$\Ue$-action on the level $l$ Fock space}\label{sec_action_fock} \

\begin{num}
For $s\in\Z$ and $N\in\Z_{>0}$,  we write $$\Z^N(s) = \left\{ (x_1,\dots,x_N)\in \Z^N \, | \, \sum_{k=1}^N x_k = s \right\}.$$
\end{num}

In all what follows, we fix $s\in\Z$, $l, e\in\Z_{\geq 2}$ and $q$ be an indeterminate. 
Set also $p=-q^{-1}$.
For each $l$-charge $\bs = (s^l_1,\dots,s^l_l)\in \Z^l(s)$, consider the level $l$ Fock space 
$$
\cF_{\bs,e} = \bigoplus_{\bla \in \Pi_l} \C(q) |\bla,\bs\rangle.
$$

The $\C(q)$-algebra $\Ue$ can be defined by a presentation by generators and relations,
where the generators are denoted $e_i,f_i,t_i$, $i=0,\dots,e-1$ and called the \textit{Chevalley generators} of $\Ue$ \cite[Definition 6.1.3]{GeckJacon2011}.
We do not recall the representation theory of $\Ue$, but simply mention that there is 
a nice module category denoted $\cO_{\mathrm{int}}$ consisting of so-called \textit{integrable} modules, see \cite[Definition 3.1.7]{GeckJacon2011}.
In particular, this category is semisimple and its objects have a so-called \textit{weight decomposition}.

\begin{thm}[\mbox{\cite{JMMO1991}}]\label{thm_action_fock}
The space $\cF_{\bs,e}$ is an integrable $\Ue$-module with
respect to the following action:
\begin{equation}\label{action_fock}
\begin{array}{ccl}
t_i |\bla,\bs\rangle 
&=&
q^{N_i(|\bla,\bs\rangle)}|\bla,\bs\rangle
\\
e_i |\bla,\bs\rangle
& = &
\dispst\sum_{\res(\bla\backslash\bmu)=i} q^{-N_i^<(|\bla,\bs\rangle,|\bmu,\bs\rangle)}|\bmu,\bs\rangle
\\
f_i |\bla,\bs\rangle
& = &
\dispst\sum_{\res(\bmu\backslash\bla)=i} q^{N_i^>(|\bmu,\bs\rangle,|\bla,\bs\rangle)}|\bmu,\bs\rangle
\end{array}
\end{equation}
\end{thm}

In the sequel, we will also use level $e$ Fock spaces
$\cF_{\dbs,l}$, where $\dbs = (\ds^e_1,\dots,\ds^e_e)\in \Z^e(s)$
is an $e$-charge.
They are endowed with the structure of an integrable $\Ul$-module via the formulas
\ref{action_fock} replacing $q$ by $p$ and exchanging $e$ and $l$.
For clarity, we might further want to use the notation
$\dot{t}_j, \dot{e}_j, \dot{f}_j$ for the generators of $\Ul$.

\subsection{Uglov's decomposition of Fock spaces}\label{sec_uglov} \

Let $\La^s$ be the level one Fock space associated to $s$, i.e.
$$\La^s=\bigoplus_{\la\in\Pi}\C(q)|\la,s\rangle.$$
In Uglov \cite{Uglov1999}, level $l$ Fock spaces $\cF_{\bs,e}$ are realised as submodules of
$\La^{s}$.
By the results of Section \ref{uglovs_bij}, there are two alternative ways to index the basis elements 
of $\La^{s}$, namely:
\begin{itemize}
 \item by charged $l$-partitions $|\bla,\bs\rangle$, via the bijection $\tau$,
 \item by charged $e$-partitions $|\dbla,\dbs\rangle$, via the bijection $\dtau$.
\end{itemize}
This yields the following correspondence:

\begin{equation}\label{indexation_ug}
\begin{array}{ccccc}
\dispst\bigoplus_{\bs\in\Z^l(s)}\cF_{\bs,e} & 
\xrightleftharpoons[\hspace{4mm}\tau\hspace{4mm}]{\tau^{-1}} & 
\La^s & 
\xrightleftharpoons[\dtau^{-1}]{\hspace{4mm}\dtau\hspace{4mm}}  & 
\dispst\bigoplus_{\dbs\in\Z^e(s)} \cF_{\dbs',l}
\\
|\bla,\bs\rangle & \longleftrightarrow & 
|\la,s\rangle &
\longleftrightarrow &
|\dbla,\dbs\rangle.
\end{array}
\end{equation}

According to \cite[Section 4.2]{Uglov1999},
there is an action of $\Ue$, of $\Ul$, and of a Heisenberg algebra $\cH$ on $\La^s$.
Moreover, the action of $\Ue$ (respectively $\Ul$)
on $\La^s$ induces,
via the indexation by $l$-partitions (respectively $e$-partitions),
an action on $\cF_{\bs,e}$ (respectively $\cF_{\dbs,l}$).
These actions are precisely the one of Section \ref{sec_action_fock}.
In fact, we have 
\begin{equation}\label{decomp_l}
\La^s = \bigoplus_{\bs\in\Z^l(s)}\cF_{\bs,e}
\text{\quad and \quad}
\La^s = \bigoplus_{\dbs\in\Z^e(s)}\cF_{\dbs,l}.
\end{equation}

\begin{num}\label{fund_dom} Denote 
$$\begin{array}{rcl} 
A(s) & = & \left\{ (s_1,\dots,s_l)\in \Z^l(s) \, | \, s_1 \leq \dots \leq s_l < s_1+e \right\} \quad\text{and}\\
\dA(s) & = & \left\{ (\ds_1,\dots,\ds_e)\in \Z^e(s) \, | \, \ds_1 \leq \dots \leq \ds_e < \ds_1+l \right\}.
  \end{array}$$
\end{num}

\begin{thm}[\mbox{\cite[Proposition 4.6 and Theorem 4.8]{Uglov1999}}] \label{thm_uglov} \
\begin{enumerate}
 \item The actions of $\Ue$, $\Ul$ and $\cH$ on $\La^s$ pairwise commute.
 \item We have the decomposition
 
 \begin{equation}\label{decomp_triple}
\La^s = \bigoplus_{\bs\in A(s)} \Ue \otimes \cH \otimes \Ul  |\bemp,\bs\rangle.
\end{equation}
\end{enumerate}
\end{thm}

\subsection{Conjugating multipartitions}\label{conjugating} \

In this section,
we modify the indexation of the basis elements of $\La^s$ by charged $e$-partitions.
For a partition $\la$, denote $\la'$ its conjugate.
Using the indexation by charged partitions,
define an anti-linear isomorphism as follows
\begin{equation}\label{conj}
\begin{array}{cccc}
\La^s & \lra & \La^{-s} \\ 
| \la,s\rangle & \longmapsto & |\la',-s\rangle \\
q & \longmapsto & q^{-1}
\end{array}
\end{equation}
This is an involution of $\bigoplus_{s\in\Z}\La^s$.
We write $u'$ for the image of $u\in \La^s$.

\begin{rem}
Since $p=-q^{-1}$, we also have that $p'=p^{-1}$.
\end{rem}

\medskip

The new indexation (to be compared with Indexation (\ref{indexation_ug})) is  given by the following procedure:

\begin{equation}\label{indexation}
\begin{array}{ccccccc}
\dispst\bigoplus_{\bs\in\Z^l(s)}\cF_{\bs,e} & 
\xrightleftharpoons[\hspace{4mm}\tau\hspace{4mm}]{\tau^{-1}} & 
\La^s & 
\xrightleftharpoons[\quad '\quad ]{'} & 
\La^{-s} & 
\xrightleftharpoons[\dtau^{-1}]{\hspace{4mm}\dtau\hspace{4mm}}  & 
\dispst\bigoplus_{\dbs\in\Z^e(s)} \cF_{\dbs',l}
\\
|\bla,\bs\rangle & \longleftrightarrow & 
|\la,s\rangle &
\longleftrightarrow & 
|\la',-s\rangle &
\longleftrightarrow &
|\dbla',\dbs'\rangle
\end{array}
\end{equation}

where $\tau$ and $\dtau$ are the isomorphisms induced by 
the bijection of Section \ref{uglovs_bij}.
Here, we have decided to use the notation $|\dbla',\dbs'\rangle$
for the charged $e$-partition in this new indexation,
so that it is compatible with that of Section \ref{uglovs_bij}.
Namely if $|\dbla,\dbs\rangle=\dtau(|\la,s\rangle)$,
then $\dbla'$ is the ``conjugate'' of $\dbla$
(that is to say, conjugate each component of
$\dbla$  and reverse it),
and $\dbs'=(-\ds_e, \dots,-\ds_2,-\ds_1)$
where $(\ds_1,\dots,\ds_e)=\dbs$.
In other words, the conjugation commutes with $\dtau$.
It also commutes with $\tau$.

We set 
$$
T = \dtau\circ(.)'\circ\tau^{-1}
\mand
\dT = T^{-1}= \tau\circ(.)'\circ\dtau^{-1}.
$$

\begin{exa}\label{exa_indexation}
Take $l=3$, $e=4$, $\bla = (5.1,3.1,1)$ and $\bs=(0,-1,1)$.
\begin{figure}[H] 
\includegraphics{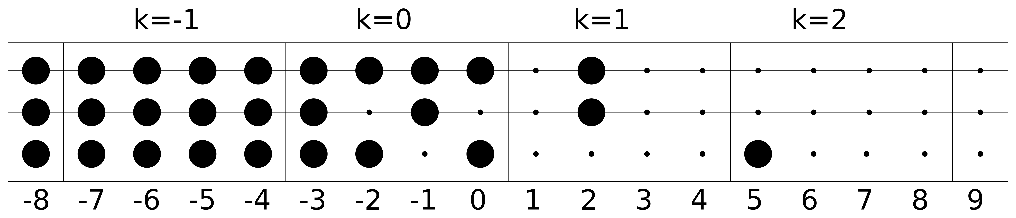}
\caption{The abacus $\cA(\bla,\bs)$}
\label{ab0}
\end{figure}
Applying the procedure described above, we get the following abacus.
\begin{figure}[H] 
\includegraphics{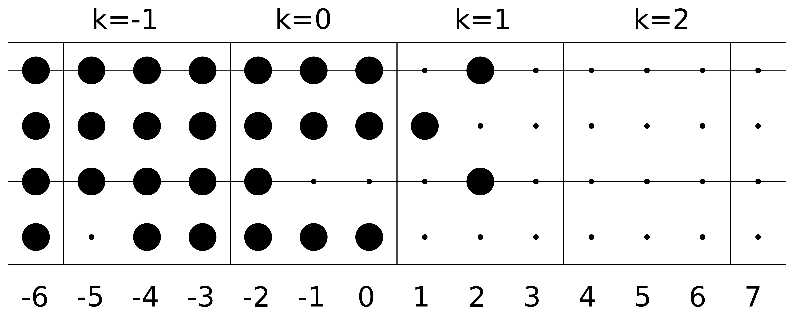}
\caption{The abacus $\cA(\dbla',\dbs')$}
\label{ab1}
\end{figure}
Hence we have $T|\bla,\bs\rangle=| \dbla',\dbs'\rangle = |(1^5,3,\emptyset,1),(-1,-1,1,1)\rangle$.
\end{exa}

\begin{prop}\label{prop_action_conj}
The action of the Chevalley operators on conjugate charged $l$-partitions
is given by the following rule.
\begin{equation}
\begin{array}{ccl}
e_{-i}|\bla',\bs'\rangle
& = &
q^{-N_{i}(|\bla,\bs\rangle)+1} (e_{i}|\bla,\bs\rangle)'
\\
f_{-i}|\bla',\bs'\rangle
& = &
q^{N_{i}(|\bla,\bs\rangle)+1} (f_{i}|\bla,\bs\rangle)',
\end{array}
\end{equation}
for all $i=0\dots,e-1$ and where indices are understood to be modulo $e$.
\end{prop}

Using Theorem \ref{thm_action_fock}, one recovers an explicit formula.
As usual, we have a similar result for the Chevalley operators
$\dot{e_j}$ and $\dot{f_j}$ of $\Ul$.

\proof
One way to see it is to use the explicit action of $\Ue$ 
in terms of removable and addable $i$-boxes
of Theorem \ref{thm_action_fock}. 
Now, it is clear that the conjugation isomorphism (\ref{conj})
maps a removable (respectively addable) $i$-box of $|\bla,\bs\rangle$
to a removable (respectively addable) $(-i)$-box of $|\bla',\bs'\rangle$,
and that it reverses the way these boxes are ordered.
The result follows.
\endproof

Note that the power of $q$ appearing can be interpreted as the action of the
element $t_i\in\Ue$ according to Theorem \ref{thm_action_fock}.
This way, we recover the claim of \cite[Proposition 5.10]{Uglov1999}.

\medskip

\begin{thm}\label{thm_uglov_bis}
The claim of Theorem \ref{thm_uglov} is also valid when the action
of $\Ul$ is computed with respect to the indexation (\ref{indexation}).
\end{thm}

\proof 
Theorem \ref{thm_uglov} says that the Chevalley operators of $\Ue$ and $\Ul$ commute on $\La^s$
when computed on the basis elements via
the correspondence $|\bla,\bs\rangle \leftrightarrow |\dbla,\dbs\rangle$.
Because of the formulas of Proposition \ref{prop_action_conj}, they still commute when we use the correspondence
$|\bla,\bs\rangle \leftrightarrow |\dbla',\dbs'\rangle$.
\endproof

\section{Two commuting crystals}\label{double_crystal}

\subsection{Reminders on crystal bases of integrable $\Ue$-representations}\label{crystal_bases}\

Kashiwara's theory of crystal bases \cite{Kashiwara1990},
yielding the theory of canonical bases, 
provides a combinatorial tool to study integrable $\Ue$-modules, by considering bases ``at $q=0$''.

Introducting crystals requires defining crystal operators \cite[Section 2.4]{Kashiwara1991}:
\begin{equation}\label{def_crystal_op}
\begin{array}{c}
\begin{array}{lcl}
\te^{}_i & = & (q t_i \De_i)^{-1/2}e_i \\
\tf^{}_i & = & (q t_i^{-1} \De_i)^{-1/2}f_i \\
\end{array}
\\
\text{where } \De_i=q^{-1}t_i+qt_i+(q-q^{-1})^2e_if_i-2,
\end{array}
\end{equation}
and where $e_i,f_i,t_i$, $i=0,\dots,e-1$ are the Chevalley operators of $\Ue$.

\medskip

Let $A_q$ (resp. $A_{q^{-1}}$) $\subset \Q(q)$  be the ring of rational functions in $q$ without pole at $0$ (resp. at $\infty$).
Following \cite[Definition 3.1.2]{Kashiwara1993}, 
we define a \textit{crystal lattice} $L$ at $q=0$ of an integrable $\Ue$-module $M$  as a free $A_q$-submodule of $M$ such that:
\begin{itemize}
 \item $L$ generates $M$ as a $\Q(q)$-vector space,
\item $L$ decomposes as a direct sum according to the weight space decomposition of $M$, 
\item $L$ is stable by the action of the crystal operators $\te_i^{}$ and $\tf_i^{}$ for all $i=0,\dots,e-1$.
\end{itemize}
We get similars definitions by replacing ``$q=0$'' by ``$q=\infty$'' and $A_q$ by $A_{q^{-1}}$.

Since a crystal lattice $L$ at $q=0$ is stable by the actions of $\te_i^{}$, $\tf_i^{}$
these induce operators on $L/qL$.
We denote them by the same symbols.
This is also valid at $q=\infty$.

\begin{defi}\label{def_crystal_basis}
A \textit{crystal basis} at $q=0$ of an integrable $\Ue$-module $M$ is a pair $(L,B)$ such that
\begin{itemize}
 \item $L$ is a crystal lattice of $M$ at $q=0$,
 \item $B$ is a $\Q$-basis of $L/qL$,
 \item $B$ decomposes as a disjoint union according to the direct sum decomposition of $L$,
 \item $\te_i B \subset B\sqcup\{0\}$ and $\tf_i B \subset B\sqcup\{0\}$ for all $i$,
 \item For all $b,b'\in B$ and for all $i$, we have $\tf_i b = b'$ if and only if $b=\te_i b'$.
\end{itemize}
\end{defi}

Similarly, one defines \textit{crystal bases at $q=\infty$}.

Crystal bases of integrable $\Ue$-modules always exist 
and are uniquely determined up to isomorphism \cite[Theorems 2 and 3]{Kashiwara1991}, \cite[Theorem 2.2]{JMMO1991}.
Therefore, we will allow the terminology ``the'' crystal basis for an integrable $\Ue$-module.

If $M$ is an integrable $\Ue$-module with crystal basis $(L,B)$, we define
the \textit{crystal graph} or \textit{$\Ue$-crystal}\footnote{
This simplified terminology is justified by the fact that the crystal graph of an integrable $\Ue$-module is an object in the category of crystals 
as defined by Kashiwara \cite[Section 7.2]{Kashiwara1995}.}  
of $M$ to be the colored oriented graph with set of vertices $B$
and arrows $b\overset{i}{\lra}b'$ whenever $b'=\tf_i b$.

\subsection{Crystal of the Fock space}\label{crystal_graph}\

We focus now on the integrable $\Ue$-module $\cF_{\bs,e}$.
The following result is \cite[Theorem 3.7]{JMMO1991}.

\begin{thm}\label{thm_crystal_basis} 
Set 
\begin{equation*}
\begin{array}{ccl}
L
&=& \dispst
\bigoplus_{\bla\in\Pi_l} A_q |\bla,\bs\rangle \text{\quad and}
\\
B
&=&
\left\{ |\bla,\bs\rangle \mod qL \, ; \, \bla\in\Pi_l \right\}.
\end{array}
\end{equation*}
Then $(L,B)$ is the crystal basis of $\cF_{\bs,e}$ at $q=0$.
\end{thm}

Thanks to this theorem, we can identify
the set of charged $l$-partitions (which is the standard basis of $\cF_{\bs,e}$)
with $B$.
We will do so in the rest of the paper.

In order to describe the crystal graph of $\cF_{\bs,e}$ combinatorially, we need to indroduce
the notion of \textit{good} boxes for $l$-partitions.
Fix $i\in\{0,\dots,e-1\}$, and let $|\bla,\bs\rangle$ be a charged $l$-partition.
Recall that we have defined in Section \ref{add_rem_boxes} a total order $<$ on the set of removable and addable $i$-boxes of 
$|\bla,\bs\rangle$.
List the addable and removable $i$-boxes of $|\bla,\bs\rangle$ in increasing order with respect to $<$,
and encode each addable $i$-box by a sign $+$ and each removable $i$-box by a sign $-$.
This yields a word in the letters $+$ and $-$, denoted $w_i(|\bla,\bs\rangle)$ 
(or simply $w_i$) and called the \textit{$i$-word} of $|\bla,\bs\rangle$.
Now, delete recursively the subwords of the form $(-+)$ in $w_i$, in order
to obtain a word of the form $(+)^\al(-)^\be$, denoted 
$\hw_i(|\bla,\bs\rangle)$ (or simply $\hw_i$) and called the \textit{reduced $i$-word} of $|\bla,\bs\rangle$.

\begin{defi}
The \textit{good} addable (respectively removable)
$i$-box of $|\bla,\bs\rangle$ is the box corresponding to the leftmost sign $-$ 
(respectively the rightmost sign $+$) in $\hw_i$. 
\end{defi}

\begin{thm}[\mbox{\cite[Theorem 3.8]{JMMO1991}}]\label{thm_crystal_graph}
We have $|\bla,\bs\rangle \overset{i}{\lra}|\bmu,\bs\rangle$ in the crystal graph of $\cF_{\bs,e}$ at $q=0$
if and only if $\bmu$ is obtained from $\bla$
by adding its good addable $i$-box (if it exists).
\end{thm}

\begin{exa}\label{exa_crystal_op}
We look again at Example \ref{exa_indexation}.
The reduced $i$-words for $i=0,1,2,3$ are
\begin{equation*}\label{iwords}
\begin{array}{ccl}
\hw_0 & = & ++- \\
\hw_1 & = & +- \\
\hw_2 & = & ++ \\
\hw_3 & = & +-
\end{array}
\end{equation*}
The action of the crystal operators is then depicted
in the following abacus:
\begin{figure}[H] 
\includegraphics{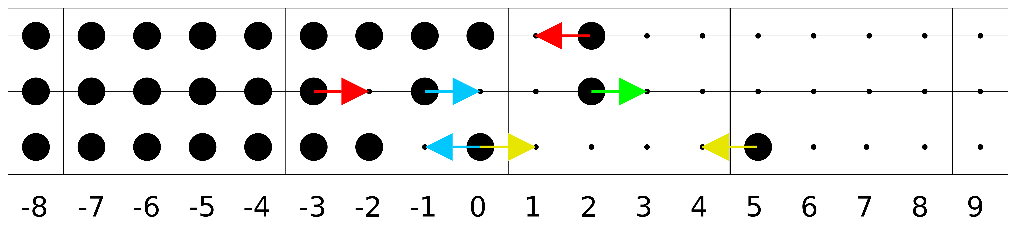}
\caption{The action of $\te_0$ and $\tf_0$ (yellow),
$\te_1$ and $\tf_1$ (red),
$\tf_2$ (green),
$\te_3$ and $\tf_3$ (blue)
}
\label{ab14}
\end{figure}

\end{exa}

By level-rank duality, one can switch the roles of $e$ and $l$
to describe the crystal graph at $p=0$ of the representations
$\cF_{\dbs,l}$ or $\cF_{\dbs',e}$ of $\Ul$.
The crystal operators appearing are denoted by ${\tde_j}$ and ${\tdf_j}$.

\begin{rem}
Because of the combinatorial definition of $\dtau$, one sees that the action of a crystal operator of $\Ul$
on the corresponding partition is to remove/add a ``good'' $l$-ribbon.
The definition of $\tau$ being different, this does not hold for $\Ue$.
\end{rem}

\subsection{Commutation of the crystal operators and formulas for computation}\label{comm_crystals}\

In this section, we compare the $\Ue$-crystal on $\cF_{\bs,e}$ and the $\Ul$-crystal on $\cF_{\dbs',l}$
using Correspondence (\ref{indexation}).
First, we need to the following fact about the corresponding crystal operators.

\begin{lem}\label{lem_comm_crystals}
The crystal operators $\te_i, \tf_i$, $i=0,\dots,e-1$ and $\tde_j,\tdf_j$, $j=0,\dots,l-1$ commute on $\La^s$.
\end{lem}

\proof
This is immediate since these operators are defined as certain combinations of the Chevalley operators, see Formulas (\ref{def_crystal_op}),
which commute by Theorem \ref{thm_uglov_bis}.
\endproof

\begin{thm}\label{thm_comm_crystal}\
\begin{enumerate}
\item  The conjugate of the crystal basis $(\dL,\dB)$ of $\cF_{\dbs,l}$ at $p=\infty$
is the crystal basis $(\dL',\dB')$ of $\cF_{\dbs',l}$ at $p=0$.
\item Correspondence (\ref{indexation_ug}) identifies $L$ with $\dL$ (respectively $\dL'$) and $B$ with $\dB$.
Under this identification, the crystal operators $\te_i, \tf_i$, $i=0,\dots,e-1$ and $\tde_j,\tdf_j$, $j=0,\dots,l-1$ commute on $L/qL$.
\end{enumerate}
\end{thm}

\proof 
By Theorem \ref{thm_crystal_basis},
the crystal basis $(\dL,\dB)$ at $p=\infty$ of $\cF_{\dbs,e}$ and
the crystal basis $(\dL',\dB')$ at $p=0$ of $\cF_{\dbs',e}$ are given by 
$$\begin{array}{cclccl}
\dL
&=& \dispst
\bigoplus_{\dbla\in\Pi_e} A_{p^{-1}} |\dbla,\dbs\rangle 
&
\dL'
&=& \dispst
\bigoplus_{\dbla'\in\Pi_e} A_{p^{}} |\dbla',\dbs'\rangle 
\\
\dB
&=&
\left\{ |\dbla,\dbs\rangle \mod p^{-1}\dL \, ; \, \dbla\in\Pi_e \right\} \quad\quad
&
\dB'
&=&
\left\{ |\dbla',\dbs'\rangle \mod p^{}\dL' \, ; \, \dbla'\in\Pi_e \right\}.
\end{array}
$$
One sees that $(\dL',\dB')$ is obtained $(\dL,\dB)$ by applying the conjugation isomorphism, proving the first claim.

Now by definition, Correspondence (\ref{indexation_ug}) sends $L=\bigoplus_{\bla\in\Pi_l} A_{q^{}} |\bla,\bs\rangle$ 
to $$\sL:=\bigoplus_{\dbla\in\Pi_e} A_{q^{}} |\dbla,\dbs\rangle.$$
Because $q=-p^{-1}$, we get $\sL=\dL$.
Similarly, $B$ is sent to $\dB$. 
Identifying $L/qL$ with $\dL/p^{-1}\dL$, the rest of the second claim is straightforward from Lemma \ref{lem_comm_crystals}.
\endproof

Of course, we could have stated the same results with crystal bases at $q=\infty$ (or equivalently at $p=0$) instead.
We decide to work at $q=0$ in the following, and use the crystal bases above.

\medskip

By Theorem \ref{thm_crystal_graph}, we can compute crystal graphs of integrable $\Ul$-modules at $p=0$.
Combining (1) and (2) from the previous theorem
shows that one needs to first
use the formula of \ref{thm_crystal_graph} for the $\Ul$-module $\cF_{\dbs',e}$, and then
use Correspondence (\ref{indexation}) (that is, the conjugation isomorphism in addition to Correspondence (\ref{indexation_ug}))
in order to compute the action of the crystal operators of $\Ul$ on $L/qL$.
This is what we will do in the rest of this paper, and what we illustrate in the following example.

\begin{exa}\label{exa_comm_crystals}
We take the same values as in Example \ref{exa_indexation}.
The action of the different crystal operators of $\Ul$ on $|\dbla',\dbs'\rangle$
is depicted on its abacus as follows:
\begin{figure}[H]
\includegraphics{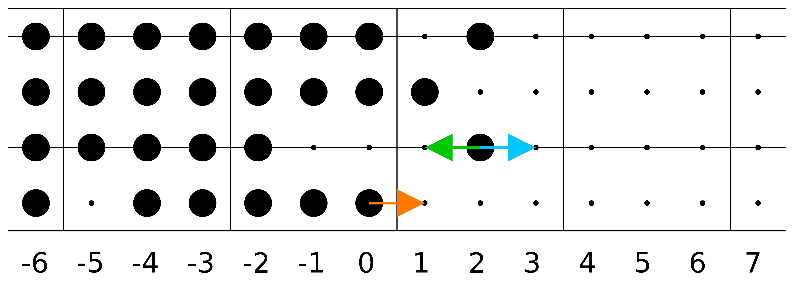}
\caption{The action of $\tdf_0$ (orange),
$\tde_1$ (green) and $\tdf_2$ (blue) on
$\cA(\dbla',\dbs')$}
\label{ab11}
\end{figure}
On the $l$-abacus representing $|\bla,\bs\rangle$,
this gives the following picture:
\begin{figure}[H] 
\includegraphics{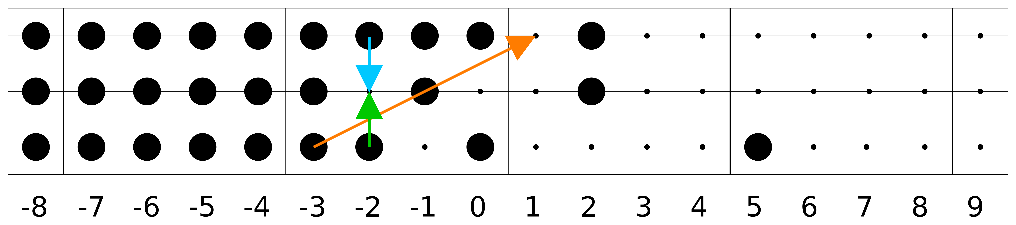}
\caption{The action of $\tdf_0$ (orange),
${\tde_1}$ (green) and ${\tdf_2}$ (blue) on
$\cA(\bla,\bs)$}
\label{ab12}
\end{figure}
Take for instance $|\bmu,\br\rangle = {\tdf_0}|\bla,\bs\rangle$.
\begin{figure}[H] 
\includegraphics{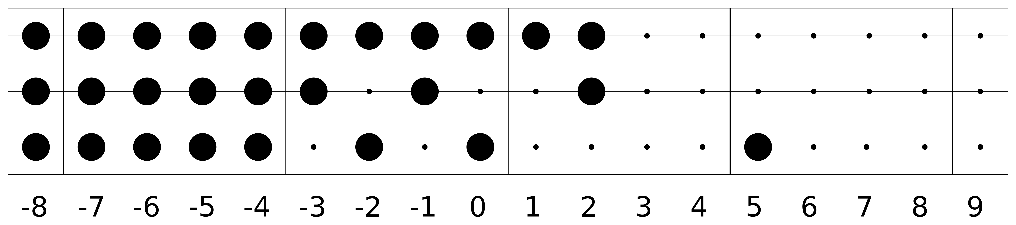}
\caption{The abacus representing the 
$l$-partition $|\bmu,\br\rangle = {\tdf_0}|\bla,\bs\rangle$}
\label{ab13}
\end{figure}
Then the reduced $i$-words ($i=0,1,2,3$) for $|\bmu,\br\rangle$ are
\begin{equation}
\begin{array}{ccl}
\hw_0 & = & ++- \\
\hw_1 & = & +- \\
\hw_2 & = & ++ \\
\hw_3 & = & +-
\end{array}
\end{equation}
These are exactly the $i$-words for $|\bla,\bs\rangle$,
see Example \ref{exa_crystal_op}.
\end{exa}

\begin{rem}\label{rem_comm_crys}
For $l=2$, the action of a dual crystal operator $\tde_j$
is a particular case of \textit{elementary operations}
in the sense of \cite[Section 7.3]{GerberHissJacon2015}.
There is another already known combinatorial procedure resembling the action of $\tde_j$,
namely Tingley's tightening procedure on abaci \cite[Definition 3.8]{Tingley2008}.
The relationship between Tingley's results and the present results will be explained in detail in
the forthcoming \cite{Gerber2016a}.
\end{rem}

\section{Doubly highest weight vertices}\label{doubly_hw}

\subsection{A combinatorial characterisation of the $\Ue$-highest weight vertices}\label{doubly_hw_char}\

According to \cite{JaconLecouvey2012}, the highest weight vertices for the $\Ue$-crystal structure are
precisely the charged $l$-partitions whose abacus is totally periodic.
Respectively, the same holds for the $\Ul$-crystal structure and $e$-partitions.
Let us recall the notion of totally periodic multipartition, cf \cite[Definition 2.2]{JaconLecouvey2012}.

Consider the abacus $\cA$ representing a charged multipartition $|\bla,\bs\rangle$.
The first \textit{$e$-period} in $\cA$ is, if it exists, the sequence 
$$P=((j_1,\be_1), \dots, (j_e,\be_e))$$
of $e$ beads in $\cA$ such that
\begin{itemize}
 \item $\be_1$ is the greatest $\be$-number appearing in $\cA$,
 \item $\be_i = \be_{i-1}-1$ for all $i=2,\dots,e$,
 \item $j_i \leq j_{i-1}$ for all $i=2,\dots,e$,
 \item for all $i=1,\dots,e$, there does not exist $(j_0,\be_i)\in\cA$ such that $j_0\leq j_i$.
\end{itemize}

The first period of $\cA \backslash P$, if it exists, is called the second period of $\cA$.
We define similarly the $k$-th period of $\cA$ by induction.

The abacus $\cA$ is said to be \textit{totally $e$-periodic} if it has 
infinitely many $e$-periods.
In this case, there exists a non-negative integer $N$ such that the abacus obtained 
from $\cA$ by removing its first $N$ periods corresponds to the empty multipartition.
We call an $e$-period $P$ \textit{trivial} if 
$$(j,\be)\in P \quad \Ra \quad (j,\be-c)\in\cA \text{ for all } c\in\Z_{>0}.$$
In other words, a period is trivial if it encodes only size zero parts.

The key property of periods is the following:

\begin{prop}\label{prop_period}
An $e$-period does not contribute to the computation of the 
reduced $i$-words (for all $i=0,\dots,e-1$).
\end{prop}

\begin{proof}
Let $P=((j_1,\be_1), \dots, (j_e,\be_e))$ be the first $e$-period in $\cA=\cA(\bla,\bs)$.
Fix a residue $i\in\{0,\dots,e-1\}$ and look at the $i$-word (respectively reduced $i$-word) $w_i$
(respectively $\hw_i$) for $\cA$ on the one hand,
and the $i$-word (respectively reduced $i$-word) $v_i$
(respectively $\hv_i$) for $\cA\backslash P$ on the other hand.
Let us show that $\hw_i=\hv_i$.

Suppose first that $i=\be_1\mod e$.
Then $(j_1,\be_1)$ corresponds to a rightmost sign $+$ in $w_i$.
Now, either $(j_e,\be_e-1)\in\cA$, in which case $(j_e,\be_e-1)\in\cA\backslash P$,
$(j_e,\be_e-1)$ corresponds to a rightmost sign $+$ in $v_i$;
or $(j_e,\be_e-1)\notin\cA$, in which case $(j_e,\be_e)$ corresponds to a sequence $(-+)$
in $w_i$, which simplifies in $\hw_i$.
So in both cases, $\hw_i=\hv_i$.

Suppose now that $i\neq\be_1\mod e$.
Then there is an element $(j_{k_0},\be_{k_0})\in P$ such that $i=\be_{k_0}\mod e$.
If $j_{k_0+1} = j_{k_0}$, this means that $(j_{k_0},\be_{k_0})$ and $(j_{k_0+1},\be_{k_0+1})$
do not contribute to the computation of $w_i$, and therefore neither in $\hw_i$.
In the case where $j_{k_0+1} < j_{k_0}$, then either
$(j_{k_0},\be_{k_0}-1)\in \cA$, in which case there is a rightmost $+$ in $w_i$, corresponding to $(j_{k_0+1},\be_{k_0+1})$,
which also exists in $v_i$, corresponding to $(j_{k_0},\be_{k_0}-1)\in P$;
or $(j_{k_0},\be_{k_0}-1)\notin \cA$, in which case there is a sequence $(-+)$ in $P$,
corresponding to the beads $(j_{k_0},\be_{k_0})$ and $(j_{k_0+1},\be_{k_0}-1)$,
which simplifies in $\hw_i$.
Again, in both cases, $\hw_i=\hv_i$.
\end{proof}

\begin{thm}[\mbox{\cite[Theorem 5.9]{JaconLecouvey2012}}]\label{thm_hw_vertices}
The charged $l$-partition $|\bla,\bs\rangle$ is a highest weight vertex
in the $\Ue$-crystal if and only if $\cA(\bla,\bs)$ is totally $e$-periodic.
\end{thm}

This result also holds by switching $e$ and $l$ and replacing $q$ by $p$.

\begin{cor} The charged partition $|\la,s\rangle$ is a
highest weight vertex in both the $\Ue$-crystal and the $\Ul$-crystal
if and only if:
\begin{enumerate}
 \item $\cA(\bla,\bs)$ is totally $e$-periodic, and
 \item $\cA(\dbla',\dbs')$ is totally $l$-periodic.
\end{enumerate}
Such an element is called a \textit{doubly highest weight vertex}.
\end{cor}

\proof
This is a direct consequence of Theorem \ref{thm_hw_vertices}
together with Theorem \ref{thm_comm_crystal}.
\endproof

\begin{exa} \label{exa_doubly_hw}
Take $l=3$, $e=2$, $\bla = (3,3.1,1)$ and $\bs=(-1,0,0)$.
Then we have $\dbla' = (2^2.1^3,2.1^3)$ and $\dbs' = (0,1)$.
\begin{figure}[H] 
\includegraphics{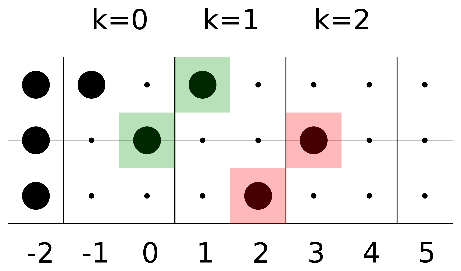}
\caption{The abacus $\cA(\bla,\bs)$}
\label{ab9}
\end{figure}
One sees that $\cA(\bla,\bs)$ has two $e$-periods, and is totally $e$-periodic.
The first period corresponds to parts of size $3$ (colored in red)
and the second to parts of size $1$ (colored in green).

\begin{figure}[H] 
\includegraphics{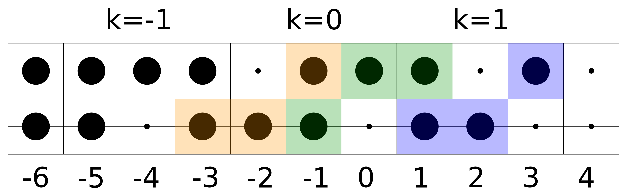}
\caption{The abacus $\cA(\dbla',\dbs')$}
\label{ab10}
\end{figure}
Similarly, $\cA(\dbla',\dbs')$ has three $l$-periods, and is totally $l$-periodic.
The first period corresponds to parts of size $2$ (blue) and the next two to parts of size $1$ (green, orange).
The other periods are trivial.

Therefore, the associated charged partition $|\la,s\rangle = |(10.8.4.2),-1\rangle$
is a doubly highest weight vertex.
\end{exa}

\subsection{Properties of doubly highest weight vertices}\label{doubly_hw_properties}\

We now list some properties of such charged partitions.
In what follows, we let $|\la,s\rangle$ be a doubly highest weight vertex.

\begin{lem}\label{lem_double_2}\
\begin{enumerate}
 \item In $\cA(\bla,\bs)$ (respectively $\cA(\dbla',\dbs')$),
 all beads of a given period correspond to parts of the same size.
 We then denote $S_k$ (respectively $\dS_k$) the part size corresponding to the $k$-th period
 in $\cA(\bla,\bs)$ (respectively $\cA(\dbla',\dbs')$).
 \item Let $N$ (respectively $\dN$) be the number
 of non-trivial periods in $\cA(\bla,\bs)$ (respectively $\cA(\dbla',\dbs')$).
 We have $\#\{ S_k \, ; \, 1\leq k \leq N \} = \#\{ \dS_k \, ; \, 1\leq k \leq \dN \}$.
\end{enumerate}
\end{lem}

\proof
We prove this by induction on $n=N$. Write $\cA=\cA(\bla,\bs)$.
Suppose first that $n=0$.
Then we also have $\dN=0$ since 
$|\dbla',\dbs'\rangle$ is totally $l$-periodic, and both statements hold.
Let now $n>0$ and suppose that the lemma is true for all charged multipartitions
whose number of non-trivial $e$-periods is smaller than $n$.
Since $n\geq 1$, there is at least a position $(j,c)\notin\cA$
such that $(j,c+1)\in\cA$. 
Consider the position $(j_0,c_0)$ verifying this property, such that $c_0$ is maximal and such that
$j_0$ is maximal among $\{ (j,c_0)\notin\cA \, | \, (j,c_0+1)\in\cA \}$
Then $(j_0,c_0+1)\in\cA$ and therefore belongs to an $e$-period $P$ of $\cA$, since
$\cA$ is totally $e$-periodic.
Let $Q=((j_1,c_1),\dots, (j_r,c_r))$ be the sequence of empty positions in $\cA(\bla,\bs)$
located directly to the left of the beads of $P$.
In particular, $(j_1,c_1)=(j_0,c_0)$.
Since $\dot{\cA}=\cA(\dbla',\dbs')$ is totally $l$-periodic
and since all elements of $Q$ correspond to beads in $\dot{\cA}$, 
$Q$ corresponds to a whole $l$-period in $\dot{\cA}$,
$r=l$ and $j_k=l-k+1$ for all $k=1,\dots l$.
Similarly, if there is an empty spot directly to the left of a position in the set $Q_1=Q$,
this yields a set $Q_2$ of $l$ beads corresponding to an $l$-period in $\dot{\cA}$.
Iterating, we get sets $Q_k$ for $k=1,\dots,t$ where $t$ is the part size encoded by, say, 
the first bead of $P$.
Consider the abacus $\cA'=\cA\backslash (P_1\cup\dots\cup P_u)$ where
$P_k$ is the $k$-th $e$-period of $\cA$ and $P_u=P$.
This is mapped to $\dot{\cA}'=\dot{\cA}\backslash(Q_1\cup\dots\cup Q_t)$ by Correspondence (\ref{indexation}),
and therefore is again a doubly highest weight vertex, whose
number of non-trivial $e$-periods is smaller than $n$.
By induction hypothesis, both points of the lemma hold for $\cA'$ and $\dot{\cA}'$.
It is straightforward that adding $P_1\cup\dots\cup P_u$ to $\cA'$ 
(which corresponds to addin $Q_1\cup\dots\cup Q_t$ to $\dot{\cA}'$)
preserves these properties, thus proving the claim for $|\bla,\bs\rangle$.
\endproof

\begin{cor}\label{cor_double_3}\
\begin{enumerate}
\item The multiplicity of each part in $\bla$ (respectively $\dbla'$) is divisible by $e$ (respectively $l$).
\item The rank of $\la$ is divisible by $el$.
\end{enumerate}
\end{cor}

\proof \
\begin{enumerate}
\item This is straightforward from Lemma \ref{lem_double_2} (1), since
each part of $\bla$ (respectively $\dbla'$)
is read off the abacus by looking at each bead of the non-trivial periods.
\item By Point (1) above, the multiplicity of each part in $\dbla'$ is divisible by $l$.
Using Remark \ref{rem_ribbon} and the definition of the correspondence \ref{indexation}, 
we can see $\la'$ as the partition with $e$-quotient $\dbla'$ and $e$-core determined by $\dbs'$.
Therefore the rank of $\la'$ is divisible by $el$, hence so is that of $\la$.
\end{enumerate}
\endproof

Recall that we have defined the domains $A(s)$ and $\dA(s)$ in Notation \ref{fund_dom}.

\begin{prop}\label{prop_charge}
We have
$$ \bs\in A(s) \mand
\dbs'\in \dA(s).
$$
\end{prop}

\proof
Recall that the multicharge is read from the abacus by
shifting all beads to the left and looking at the index of the rightmost bead in each row
of the resulting abacus.
By Lemma \ref{lem_double_2}, all periods in a doubly highest weight vertex correspond to the same part.
Therefore, a doubly highest weight charged $l$-partition $|\bla,\bs\rangle$  is obtained from 
$|\bemp,\bs\rangle$ by shifting whole $e$-periods to the right.
This is a key fact and will be used in what comes next.
Therefore, it suffices to prove the claim for vertices of the form $|\bla,\bs\rangle$.
In this case,
let us observe the corresponding charged $e$-partition $|\dbla',\dbs'\rangle$ defined via Formula (\ref{indexation}).
By definition of $T$, if $\bs\notin A(s)$, then: 
\begin{itemize}
\item either there exists and index $j$ such that $s_j>s_{j+1}$, in which case the difference $\de=s_j-s_{j+1}$
creates $\de$ empty spots in the $l$-abacus which are beads in the $e$-abacus but do not form an $l$-period,
\item or there exists two indices $j$ and $j'$ such that $j'<j$ and $s_{j'}<s_j-e$, in which case
the difference $s_j-s_{j'}$ also gives beads in the corresponding $e$-abacus which do not form an $l$-period.
\end{itemize}
In both cases, $\cA(\dbla',\dbs')$ is not totally $l$-periodic, which is a contradiction, so the claim is proved.
\endproof

\subsection{Shifting periods in abaci}

We now consider the crucial procedure of shifting periods one step to the left.
Let $P$ be an $e$-period in $\cA(\bla,\bs)$ which is shiftable one step to the left.
By Lemma \ref{lem_double_2} above, it is equivalent to say that there exists
$(j,\be)\in P$ such that $(j,\be-1)\notin\cA(\bla,\bs)$ 
(i.e. there is one empty spot left adjacent to some bead in $P$).
Then by the same observation as in the proof of Lemma \ref{lem_double_2},
there is a corresponding $l$-period $P'$ in $\cA(\dbla',\dbs')$ which is shiftable
one step to the left.
Define then
\begin{equation}\label{shift}
\begin{array}{cccc}
\varphi_{P}: &
\cA(\bla,\bs) & \lra &
\{1,\dots,l\}\times\Z \\
& (j,c) & \longmapsto &
\left\{
\begin{array}{ll}
(j,c-1) & \text{ if } (j,c)\in  P \\
(j,c) & \text{ otherwise.}
\end{array}
\right.
\end{array}
\end{equation}
The image of $\cA(\bla,\bs)$ under $\varphi_{P}$ is the $l$-abacus obtained from 
$\cA(\bla,\bs)$ by shifting $P$ one step to the left.
We define $\dvarphi_{P'}$ similarly, that is to say, the map shifting $P'$ one step to the left
in the $e$-abacus $\cA(\dbla',\dbs')$.

\begin{lem}\label{lem_double_4}\
We have
$$\varphi_{P} = \dT\circ\dvarphi_{P'}\circ T.
$$
\end{lem}

\proof
Assume, without loss of generality, that $l\leq e$.
We have already explained how an $l$-period $P'$ corresponds to a given $e$-period $P$.
Now, write
$$ P = \left\{ (j_k,\be_k) \, ; \, k\in\{1,\dots,e\} \right\} \subseteq \cA(\bla,\bs).$$
By definition of $\varphi_{P}$ (\ref{shift}), 
$\varphi_{P}$ only affects $P$, namely $\varphi_{P}(P)$ is an $e$-period $\hP$
defined by
$$\hP =\left\{ (j_k,\be_k-1) \, ; \, k\in\{1,\dots,e\} \right\}.$$
If $k$ is such that $j_k=j_{k+1}$ (which is a case that necessarily happens if $e>l$),
then $\be_{k+1}=\be_k-1$.
Moreover,
$$ 
\begin{array}{ccl}
\{ (j_k,\be_k) , (j_k, \be_{k+1} ) \} & \overset{\varphi_{P}}{\longmapsto} &
\{ (j_k,\be_k-1) , (j_k, \be_{k+1}-1 )\} \\
& & = \{(j_k,\be_{k+1}) , (j_k, \be_k-2 )\},
\end{array}
$$
so $\varphi_{P}$ fixes $(j_k,\be_{k+1})$.
Therefore, $\varphi_{P}$ fixes all elements of $P$ that are on the same row but one,
so it moves exactly $e$ beads,
and in fact these are the beads of $P'$ and they are moved in $\cA(\dbla',\dbs')$
one step to the left.
This is precisely the action of $\dvarphi_{P'}$ on $\cA(\dbla',\dbs')$ by definition.
\endproof

\begin{rem}
Note that the use of $T$ and its inverse only means that 
we look at the action of $\varphi_{P}$ on the $e$-abacus using Indexation \ref{indexation}.
In fact, this lemma claims that 
shifting an $l$-period of $\cA(\dbla',\dbs')$ one step to the left amounts to
shifting an $e$-period of $\cA(\bla,\bs)$ one step to the left.
Of course, the same holds by switching $\cA(\dbla',\dbs')$ and $\cA(\bla,\bs)$ and $e$ and $l$.
In particular, this procedure is always well defined when the considered period is the last non-trivial period.
This will be used in Section \ref{btilde}.
\end{rem}

\begin{exa}\label{exa_shift_period}
To illustrate the phenomenon explained in the proof of Lemma \ref{lem_double_4},
look at Example \ref{exa_doubly_hw}.
We get the following picture.
\begin{figure}[H] 
\includegraphics{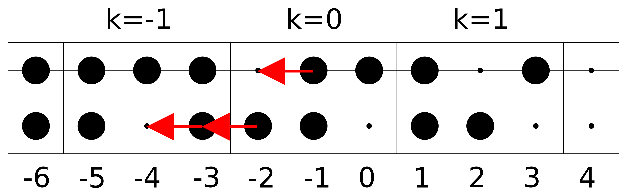}

\vspace{0.5cm}

\includegraphics{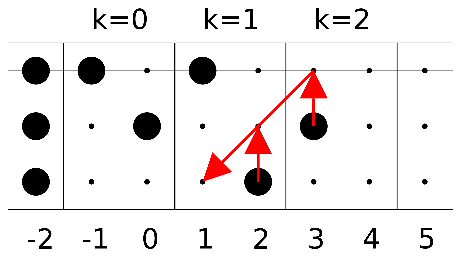}
\caption{Shifting the last non-trivial period of $\cA(\dbla',\dbs')$ one step to the left,
and its representation on $\cA(\bla,\bs)$}
\label{ab18}
\end{figure}
We see that on $\cA(\bla,\bs)$, the action depicted with the red arrows actually
corresponds to shifting a period of $\cA(\bla,\bs)$ (the first one) one step to the left.
\end{exa}

\subsection{The partition $\ka$}\

Denote $\sS = \left\{ S_k \, ; \, 1\leq k \leq N \right\}$
and $\dsS = \left\{  \dS_k \, ; \, 1\leq k \leq \dN \right\}$.
The elements of $\sS$ (respectively $\dsS$) are the different
non-zero size parts of $\bla$ (respectively $\dbla'$), see Lemma \ref{lem_double_2} (2).
Note that we have $$S_k < S_{k-1} \text{\quad for all } k\in\{ 2,\dots,N \}.$$
Similarly, $$\dS_k < \dS_{k-1} \text{ \quad for all } k\in\{ 2,\dots,\dN \}.$$
For $S_k$ in $\sS$ (respectively $\dS_k$ in $\dsS$),
denote $M(S_k)$ (respectively $\dM(\dS_k)$)
the multiplicity of the non-zero part $S_k$ (respectively $\dS_k$)
in the $l$-partition $\bla$
(respectively the $e$-partition $\dbla'$).
By Corollary \ref{cor_double_3}, $M(S_k)$ (respectively $\dM(\dS_k)$)
is divisible by $e$ (respectively $l$).
Let $m(S_k)$ (respectively $\dm(\dS_k)$) be the integer $M(S_k)/e$
(respectively $\dM(\dS_k)/l$).

Set 
\begin{equation}\label{kappa}
\begin{array}{l}
\ka = (S_1^{m(S_1)},S_2^{m(S_2)}, \dots, S_{N}^{m(S_{N})}) \\
\dka = (\dS_1^{\dm(\dS_1)},\dS_2^{\dm(\dS_2)},\dots, \dS_{\dN}^{\dm(\dS_{\dN})})
\end{array}
\end{equation}

\begin{rem}\label{rem_kappa_1}
Equivalently, 
$\ka$ can be defined as 
the ordered multiset $\left\{ S_k \, ; \, 1\leq k \leq N \right\}$
(and similarly for $\dka$).
\end{rem}

\begin{prop}\label{prop_kappa_1}
The sequences $\ka$ and $\dka$ are partitions, and $\dka = \ka'_l$.
\end{prop}

\proof
Because $S_k < S_{k-1} \text{\quad for all } k\in\{ 2,\dots,N \}$ and
$\dS_k < \dS_{k-1} \text{ \quad for all } k\in\{ 2,\dots,\dN \}$ as already observed,
$\ka$ and $\dka$ are partitions.
Further, each period $P$ in the $l$-abacus corresponds to a period $P'$ in the $e$-abacus,
and therefore the partition $\dka$ can be read off the partition $\ka$.
In fact, by definition of the correspondence $T$ (\ref{indexation}),
it is obtained by conjugating the original partition $\ka$.
\endproof

\begin{exa}\label{exa_kappa}
Take the charged multipartition in Example \ref{exa_doubly_hw}.
We have $\sS = \{ 3,1 \}$, with $M(3) = 2$ and $M(1) = 2$.
Similarly, we have $\dsS = \{ 2,1\}$ with $\dM(2) = 3$ and $\dM(1)=6$.
We get $\ka = (3.1)={\Yboxdim{5pt}\young(\,\,\,,\,)}$
and $\dka = (2.1^2)={\Yboxdim{5pt}\young(\,\,,\,,\,)}$.

Note that using the multiset definition of $\ka$ and $\dka$ (Remark \ref{rem_kappa_1}),
we have directly $\ka = \sS = \{3,1\}$ 
and $\dka = \dsS = \{ 2,1,1\}$.
\end{exa}

\begin{rem}\label{rem_kappa_2}
Note that $\ka$ depends on $|\bla,\bs\rangle$.
In fact, it induces two maps 
$$\begin{array}{cccc}
   \ka: & B & \lra & \Pi \\
        & |\bla,\bs\rangle & \longmapsto & \ka
  \end{array}$$
and 
$$\begin{array}{cccc}
   \ka': & \dB' & \lra & \Pi\\
        & |\dbla',\dbs'\rangle & \longmapsto & \ka'.
  \end{array}$$
In the rest, we want to use the notation $\ka(|\bla,\bs\rangle)$,
or simply $\ka(\bla)$ (or $\ka(\la)$).
Importantly, note that the map $\ka$ is surjective:
starting from a partition $\si$, it is easy to construct a doubly
highest weight $l$-partition (respectively $e$-partition) 
$|\bla,\bs\rangle$ (respectively $|\dbla',\dbs'\rangle$) such that
$\ka(\la)=\si$, so $\ka$ is surjective.

Moreover, if we restrict $\ka$ to the set of doubly highest weight vertices,
it is clearly injective since two doubly highest weight $l$-partitions with different $\ka$ are different.
So $\ka$ restricted to the set of doubly highest weight vertices is a bijection.
\end{rem}

We end this section on a refinement of Corollary \ref{cor_double_3}.

\begin{cor}\label{cor_kappa_2}
We have
\begin{enumerate}
\item $|\bla| = e |\ka|$ and $|\dbla'| = l |\ka|$
\item $|\la| = el |\ka|$
\end{enumerate}
\end{cor}

\proof \
\begin{enumerate}
\item The partition $\ka$ encodes the position of the non-trivial $e$-periods in $\cA(\bla,\bs)$.
Each $e$-period consists of $e$ beads, so that $|\bla| = e |\ka|$.
Similarly, $|\dbla'| = l |\ka|$.
\item As in the proof of Corollary \ref{cor_double_3}, we use Remark \ref{rem_ribbon},
which ensures that $|\la| = e |\dbla'| = el |\ka|$ by (1).
\end{enumerate}
\endproof

\section{The Heisenberg crystal}\label{heiscrys}

The aim of this section is to obtain a ``crystal version'' of Theorem \ref{thm_uglov}.
More precisely, we want to construct \textit{crystal Heisenberg operators},
such that 
\begin{enumerate}
 \item they induce maps between $\Ue$- and $\Ul$-crystals which commute with
the two kinds of crystal operators. Such maps are called \textit{double crystal isomorphisms}.
 \item every charged $l$-partition can be obtained from the empty partition
 charged by $\bs\in A(s)$ for some $s\in\Z$ by applying some sequence of
 Kashiwara crystal operators of $\Ue$ and $\Ul$ and of Heisenberg crystal operators.
\end{enumerate}

\textbf{Notation: } If $\phi: B\lra B$ is a map between crystals, then
we will also write $\phi$ for the map from $\dB'$ to $\dB'$
as well as for the map going from the set of charged partitions to itself
induced by the correspondence (\ref{indexation}).

\subsection{The maps $\tb_{-\ka}$} \label{btilde}\

\begin{defi}\label{def_btilde}
Let $|\la,s\rangle$ be a charged partition
which is a doubly highest weight vertex.
We identify $|\la,s\rangle$ with the charged  $l$-partition $|\bla,\bs\rangle$
and the charged $e$-partition $|\dbla',\dbs'\rangle$ using (\ref{indexation}).
Define 
$$
\begin{array}{ccc}
\tb_{-1}|\la,s\rangle = |\mu,s\rangle
&
\quad \mand \quad\quad
&
\tb'_{-1}|\la,s\rangle = |\nu,s\rangle
\end{array}
$$
where $|\mu,s\rangle$ (respectively $|\nu,s\rangle$) 
is identified with $|\bmu,\bs\rangle$ (respectively $|\dbnu',\dbs'\rangle$) 
using (\ref{indexation})
where 
\begin{itemize}
 \item $\cA(\dbnu',\dbs')$ is obtained from $\cA(\dbla',\dbs')$ by shifting its last non-trivial period one step to the left.
\item $\cA(\bmu,\bs)$ is obtained from $\cA(\bla,\bs)$ by shifting its last non-trivial period one step to the left.
\end{itemize}
\end{defi}

\begin{rem}\label{rem_btilde}
Remember that we had defined a map $\varphi_{P}$ in (\ref{shift}) which shifts the period $P$ one step
to the left in the $l$-abacus.
Therefore, identifying abaci and charged multipartitions, $\tb_{-1}=\varphi_{P}$ where $P$ is the last
non-trivial $e$-period of $\cA(\bla,\bs)$ (which we have noticed earlier is well-defined).
Similarly, $\tb'_{-1}=\varphi_{e,Q}$ where $Q$ is the last non-trivial $l$-period of $\cA(\dbla',\dbs')$.
\end{rem}

Note that by Remark \ref{rem_ribbon} together with Lemma \ref{lem_double_4},
both $\tb_{-1}$ and $\tb'_{-1}$ act on $|\la,s\rangle$ by removing $l$ $e$-ribbons.

\begin{exa}\label{exa_btilde}
In Example \ref{exa_shift_period}, we have depicted the action of $\tb_{-1}$ on both the $e$-abacus
and the $l$-abacus.
In terms of multipartitions, we have 
$$
\begin{array}{l}
\tb_{-1}|(3,3.1,1),(-1,0,0)\rangle = | (3,3,\emptyset),(-1,0,0) \rangle \text{ , i.e.} \\
\tb_{-1}|(10.8.4.2),-1\rangle = | (10.8),-1\rangle \vspace{2mm} \\
\end{array}
$$
and
$$\begin{array}{l}
\tb'_{-1}|(2^2.1^3,2.1^3),(0,1)\rangle = | (2^2.1,2.1^2),(0,1) \rangle \text{ , i.e} \\
\tb'_{-1}|(10.8.4.2),-1\rangle = | (6.6.4.2),-1 \rangle \\
\end{array}
$$
\end{exa}

Recall that we have denoted $B$ the crystal graph of the Fock space $\cF_{\bs,e}$
and $\dB'$ the crystal graph of the Fock space $\cF_{\dbs',l}$ in Section \ref{comm_crystals}.
We have two induced maps between crystals which we denote the same way:
\begin{equation}\label{inducedbtilde}
\begin{array}{ccccccccc}
\tb_{-1} : & B & \lra & B & \hspace{1cm} & \tb'_{-1} : & \dB' & \lra & \dB'
\\
& |\bnu,\bs\rangle & \longmapsto & \tb_{-1}|\bnu,\bs\rangle 
& \ \hspace{0.4cm}\ &
& |\dbnu',\dbs'\rangle & \longmapsto & \tb'_{-1}|\dbnu',\dbs'\rangle,
\end{array}
\end{equation}
where $\tb_{-1}(|\bnu,\bs\rangle)$ is computed as follows:
\begin{enumerate}
\item Find the highest weight vertex in the connected
component of $B$ containing $|\bnu,\bs\rangle$
(either recursively via a sequence of Kashiwara crystal operators,
or more explicitely via the algorithm exposed in \cite[Remark 6.4]{Gerber2015}).
\item Use the correspondence (\ref{indexation}) to
get an element of $\dB'$, and find
the highest weight vertex in the connected
component of $\dB'$ of this $e$-partition.
This is a doubly highest weight vertex in view of Theorem \ref{thm_comm_crystal}.
\item Apply $\tb_{-1}$ using again the correspondence (\ref{indexation}).
\item Do the reverse operations of Points (2) and (1).
The resulting $l$-partition is $\tb_{-1}|\bnu,\bs\rangle$.
\end{enumerate}
The map $\tb'_{-1} : \dB' \lra  \dB'$ is defined similarly,
switching indexations using (\ref{indexation}) and
replacing $\tb_{-1}$ by $\tb'_{-1}$ in the above procedure.
In particular $\tb_{-1}$ and $\tb_{-1}'$ are double crystal isomorphisms.

\medskip

By extension, we define $\tb_{-z}: B\lra B$ by saying that,
for a doubly highest weight vertex $|\bla,\bs\rangle\in B$, $\tb_{-z}|\bla,\bs\rangle$
is obtained from $|\bla,\bs\rangle$ by shifting its last non-trivial $e$-period $z$ steps to the left, if possible.
We extend this to $B$ in the same way as $\tb_{-1}$ (Formula (\ref{inducedbtilde}) above).
We define similarly $\tb'_{-z}: \dB'\lra \dB'$.

\begin{rem}\label{rem_btildez}
One sees that $\tb_{-z}$ is the $z$-fold composition of $\tb_{-1}$ if and only if $S_{N} = z$
(see Lemma \ref{lem_double_2}, i.e. if and only if the last non-trivial $e$-period of $\bla$ corresponds to a part $z$).
We have the similar property for $\tb'_{-z}$.
\end{rem}

\begin{defi}\label{def_btildepart}
For a partition $\si=(\si_1,\si_2,\dots,\si_t)$,
we define $\tb_{-\si} :  B   \lra  B$ and $\tb'_{-\si} :  \dB'   \lra  \dB'$
through the formulas
\begin{equation}
\begin{array}{l}
\tb_{-\si} = \tb_{-\si_1}\circ\tb_{-\si_2}\circ\dots\circ\tb_{-\si_t} \mand \\
\tb'_{-\si} = \tb'_{-\si_1}\circ\tb'_{-\si_2}\circ\dots\circ\tb'_{-\si_t}.
\end{array}
\end{equation} 
When one of the $\tb_{-\si_k}$ is not well defined on 
$(\tb_{-\si_{k+1}}\circ\dots\circ\tb_{-\si_t})|\bla,\bs\rangle$, 
we set $\tb_{-\si}|\bla,\bs\rangle=0$
(and similarly for $\tb'_{-\si}$).
\end{defi}

Remember that the partition $\ka$ (respectively $\ka'$) 
associated to $|\la,s\rangle$ (see \ref{kappa})
is written $\ka=(\ka_1,\ka_2,\dots,\ka_t)$ where each $\ka_i$ is a certain $S_{k}$
(respectively $\ka'=(\ka'_1,\ka'_2,\dots,\ka'_u)$ where each $\ka'_i$ is a certain $\dS_{k}$).
So because of Remark \ref{rem_btildez}, $\tb_{-\ka}$ and $\tb'_{-\ka'}$
are well-defined on $|\la,s\rangle$.

\begin{prop}\label{prop_btildekappa_1}\
\begin{enumerate}
 \item If $|\la,s\rangle$ is a doubly highest weight partition,
then $\tb_{-\ka}|\bla,\bs\rangle=|\bemp,\bs\rangle$
and $\tb'_{-\ka}|\dbla',\dbs'\rangle=|\dbemp,\dbs'\rangle$.
\item The following diagram commutes
$$ \xymatrix@!0 @R=1.5cm @C=2cm{
  B  \ar[r]^-{T}\ar[d]_-{\tb_{-\ka}}
&
  \dB' 
  \ar[d]^-{\tb'_{-\ka'}} 
\\
    B \ar[r]_-{T}
&
   \dB',
  }$$
where here, the partition $\ka$ depends on
the chosen multipartition (see Remark \ref{rem_kappa_2}).
Therefore, we write $\tb_{-\ka} = \tb'_{-\ka'}.$
\end{enumerate}

\end{prop}

\proof\
\begin{enumerate}
 \item By definition of $\tb'_{-\si}$ on $\dB'$
and by construction of $\ka$, it is straightforward
the $e$-partition $\tb'_{-\ka'}|\dbla',\dbs'\rangle$ is empty
(and $\tb_{-\ka'}$ does not modify the $e$-charge).
Similarly, we have $\tb_{-\ka}|\bla,\bs\rangle=|\bemp,\bs\rangle$.
\item 
Because of Proposition \ref{prop_charge},
$T$ maps $|\bemp,\bs\rangle$ to $|\dbemp,\dbs\rangle$.
In other words, the $e$-partition associated to 
the $|\bemp,\bs\rangle$ such that $\bs\in A(s)$
via the correspondence (\ref{indexation}) is also empty.
Together with Point (1), we get the commutativity of the diagram.
\end{enumerate}
\endproof

\begin{num}\label{not_dhw}
If $|\la,s\rangle$ is a doubly highest weight vertex, we will use the notation
$|\bar{\la},s\rangle$ for the charged partition $\tb_{-\ka(\la)}|\la,s\rangle$.
\end{num}

\begin{prop}\label{prop_btildekappa}
The map $\tb_{-\ka}$ is a double crystal isomorphism.
\end{prop}

\proof
This is a direct consequence of the definition of $\tb_{-\ka}$
together with the fact that $\tb_{-1}$ is a double crystal isomorphism 
and Remark \ref{rem_btildez}.
\endproof

\subsection{The inverse maps} \label{btilde_inv}\

Recall that we have introduced the notion of trivial period in Section \ref{doubly_hw_char}.

\begin{defi}\label{def_btilde+}
Let $|\la,s\rangle$ be a charged partition
which is a doubly highest weight vertex.
We identify $|\la,s\rangle$ with the charged  $l$-partition $|\bla,\bs\rangle$
and the charged $e$-partition $|\dbla',\dbs'\rangle$ using (\ref{indexation}).
Define 
$$
\begin{array}{ccc}
\tb_{1}|\la,s\rangle = |\mu,s\rangle
&
\quad \mand \quad\quad
&
\tb'_{1}|\la,s\rangle = |\nu,s\rangle
\end{array}
$$
where $|\mu,s\rangle$ (respectively $|\nu,s\rangle$) 
is identified with $|\bmu,\bs\rangle$ (respectively $|\dbnu',\dbs'\rangle$) 
using (\ref{indexation})
where 
\begin{itemize}
 \item $\cA(\dbnu',\dbs')$ is obtained from $\cA(\dbla',\dbs')$ by shifting its first trivial period one step to the right.
\item $\cA(\bmu,\bs)$ is obtained from $\cA(\bla,\bs)$ by shifting its first trivial period one step to the right.
\end{itemize}
\end{defi}

We extend this definition and write, for a positive integer $z$,
$\tb_z|\bla,\bs\rangle$ to be the $l$-partition obtained 
by shifting the first trivial period of $|\bla,\bs\rangle$ $z$ steps to the right.
Similarly, 
$\tb'_z|\dbla',\dbs'\rangle$ to be the $e$-partition obtained 
by shifting the first trivial period of $|\dbla',\dbs'\rangle$ $z$ steps to the right.
Finally, for a partition $\si=(\si_1,\dots,\si_t)$, we define 
$\tb_{\si} =  \tb_{\si_t} \circ \dots \circ \tb_{\si_1}$ and 
$\tb'_{\si} =  \tb'_{\si_t} \circ \dots \circ \tb'_{\si_1}$.
When this is not well-defined, we set again $\tb_{\si}|\bla,\bs\rangle =0$ 
(respectively $\tb'_{\si}|\dbla',\dbs'\rangle=0$.
All of these maps induce maps between crystals $B\lra B$
or $\dB'\lra \dB'$ by the procedure explained in (\ref{inducedbtilde}).
In particular, the following property is straightforward.

\begin{prop}\label{prop_btildekappa+}
The map $\tb_{\si}$ is a double crystal isomorphism.
\end{prop}

\begin{rem}\label{rem_btilde_inv}
By definition of $\ka$ in Section \ref{kappa} and Proposition \ref{prop_btildekappa_1},
it is clear that for all charged partition $|\la,s\rangle$
which is a doubly highest weight vertex,
$$|\la,s\rangle = \tb_{\ka} |\bar{\la},s\rangle.$$
So, it is enough to understand the connected components in $B$ and $\dB'$
containing $|\bemp,\bs\rangle$ and $|\dbemp,\dbs'\rangle$ respectively.
This is the case we consider in the following proposition.
\end{rem}

\begin{rem}\label{rem_btilde_inv_1}
The maps $\tb_{\si}$ and $\tb_{-\si}$ are defined so that they are 
inverse to each other, that is 
$$ \tb_{\si} \circ \tb_{-\si} = \id_{B} = \tb_{-\si} \circ \tb_{\si}$$
and  
$$ \tb'_{\si} \circ \tb'_{-\si} = \id_{\dB'} = \tb'_{-\si} \circ \tb'_{\si},$$
whenever the first identities make sense.
\end{rem}

As a consequence of Remarks \ref{rem_btilde_inv} and \ref{rem_btilde_inv_1},
we see that $\tb_{\si}|\bla,\bs\rangle \neq 0$ if and only if 
the first part of $\si$ is not greater than the last part of $\ka=\ka(|\bla,\bs\rangle)$.
In this case, we have 
\begin{equation}\label{heis_map}
\begin{array}{ccl} 
\tb_{\si}|\bla,\bs\rangle & = & (\tb_{\eta}\circ\tb_{-\ka})|\bla,\bs\rangle \\
& = & \tb_{\eta}|\bemp,\bs\rangle, 
\end{array}
\end{equation}
where $\eta$ is the partition obtained by adding the parts of $\si$ to $\ka$.

\begin{prop}\label{prop_btildekappa+_0}
For all partition $\si$, we have $\tb_{\si}|\bar{\la},s\rangle = \tb'_{\si'}|\bar{\la},s\rangle$.
\end{prop}

\proof
First of all, $|\bar{\la},s\rangle$ is a doubly highest weight vertex for all $s\in\Z$, which
ensures that $\tb_{\si}|\bar{\la},s\rangle$ and $\tb'_{\si'}|\bar{\la},s\rangle$ are well-defined.
In fact, the $l$-partition and the $e$-partition corresponding are empty by Proposition \ref{prop_btildekappa_1}, and the
$l$-charge (respectively $e$-charge) is an element of $A(s)$ (respectively $\dA(s)$)
by Proposition \ref{prop_charge}.

By Lemma \ref{lem_double_4}, the action of $\tb_{1}$ on the $l$-abacus
(shifting its first trivial $e$-period one step to the right) corresponds to shifting an
$l$-period one step to the right in the $e$-abacus.
Since this $e$-abacus corresponds to the empty $e$-partition, it has only trivial $l$-periods,
and one can only shift its first trivial $l$-period to the right.
This forces $\tb_1$ to be the same as $\tb'_1$. Hence, the result holds for $\si = (1)$.
In fact, one can look directly at the action of $\tb_z$ ($=\tb_{\si}$ with $\si=(z)$) on the empty $l$-partition.
Using the combinatorial definition of the correspondence (\ref{indexation}), one sees that moving
the first trivial $e$-period in the empty $l$-abacus $z$ steps to the right creates $z$
$l$-periods in the $e$-abacus which are obtained from the empty $e$-abacus by recursively shifting its first period
one step to the right.
In other terms, it corresponds to applying $\tb'_1\circ\dots\circ\tb'_1$ (with $z$ factors) to the empty $e$-abacus.
That is to say, $\tb_{(z)} = \tb'_{(1^z)}$. Therefore, the result holds for $\si=(z)$.
Similarly, it holds for $\si=(1^z)$.
Using the same observation, we deduce that for an arbitrary $\si$,
the map $\tb_{\si}$ acts on the empty $l$-abacus exactly like
$\tb'_{\si'}$ acts on the empty $e$-abacus, with the identification (\ref{indexation}).
\endproof

\begin{rem}\label{rem_btildekappa}
In the level 1 case, Leclerc and Thibon \cite{LeclercThibon2001} have made explicit the action
of some elements $S_\si\in\cH$, defined from the basis of Schur functions in the space of symmetric functions,
on the canonical basis of the Fock space, see \cite[Theorem 6.9]{LeclercThibon2001}.
This induces an action of $\cH$ at the combinatorial level, i.e. on the crystal on partitions:
the operator $S_\si$ acts on a partition $\la$ by adding $e$ times each part of $\si$ in $\la$.
For instance, if $e=3$,
$$
S_{\Yboxdim{5pt}\young(\,\,,\,,\,)}\; ( \; {\tiny \young(\,\,\,\,\,,\,\,\,)} \; )\; = 
{\tiny \young(\,\,\,\,\,,\,\,\,,\,\,,\,\,,\,\,,\,,\,,\,,\,,\,,\,)}.
$$
In the $1$-abacus representing $\la$, this amounts to shifting 
recursively the first trivial $e$-period $\si_k$ steps to the right,
where $\si=(\si_1,\si_2,\dots)$.
So this is exactly the same procedure as our map $\tb_{\si}$.
Hence, these maps $\tb_{\si}$ can be interpreted as generalisations of the operators $S_\si$
coming from the action of $\cH$ in the level 1 case.
However, throughout this paper, $l=1$ is not allowed. In fact, in the level $1$ case, the structure of the Fock space
is somewhat different since there is only one quantum group and the Heisenberg algebra $\cH$ acting.
\end{rem}

\subsection{Definition of the Heisenberg crystal}\label{heis_crys}\

We can now define an oriented colored graph structure on the set of charged partitions, 
by setting $|\la,s\rangle \overset{c}{\lra} |\mu,s\rangle$ if 
$\ka(\mu)$ is obtained from $\ka(\la)$ by adding a box $(a,b)$ such that $b-a=c$.
As usual, we define it on doubly highest weight vertices and we extend it as in (\ref{inducedbtilde}).
We call it the \textit{Heisenberg crystal}, or simply the $\cH$-crystal, of $\La^s$.

\begin{rem}\label{rem_heis_crys}
The rule for drawing an arrow in the Heisenberg crystal is in fact
the $\Uinf$-crystal graph rule on 
$\left\{ |\ka(\la),0 \rangle \, ; \, \la\in\Pi \right\},$
which is equal to 
$\left\{ |\si,0 \rangle \, ; \, \si\in\Pi \right\}$
by the surjectivity of $\ka$ explained in Remark \ref{rem_kappa_2}.
Hence, one can see the Heisenberg crystal as
the preimage under the map $\ka$ of the $\Uinf$-crystal
on the set of partitions,
which justifies the terminology ``crystal''.
\end{rem}

\medskip

Now, observe that the procedure $|\la,s\rangle \overset{c}{\lra} |\mu,s\rangle$ is in fact a composition of 
maps $\tb_{\pm\si}$, namely
$$
|\la,s\rangle \xrightarrow[]{\quad\tb_{-\ka(\la)}\quad} |\bar{\la},s\rangle
\xrightarrow[]{\quad\tb_{\ka(\mu)}\quad}|\mu,s\rangle.
$$
This is a generalisation of Formula (\ref{heis_map}).

Therefore, we call the map
\begin{equation}\label{heis_op}
\tb_{1,c} = \tb_{\ka(\mu)} \circ \tb_{-\ka(\la)}
\end{equation}
\textit{Heisenberg crystal operator},
and there is an arrow $|\la,s\rangle \overset{c}{\lra} |\mu,s\rangle$ in the Heisenberg crystal
if and only if $|\mu,s\rangle = \tb_{1,c} |\la,s\rangle$.
This is an analogous result to Theorem \ref{thm_flotw}, in the sense that
the Heisenberg crystal graph is explicitely described in combinatorial terms (via
an explicit formula of the Heisenberg crystal operator).
In fact, $\tb_{1,c}$ is an analogue for $\cH$ of the Kashiwara crystal operator $\tf_i$
for $\Ue$.

\begin{rem}\label{rem_heis_op} \
\begin{enumerate}
 \item The map $\tb_{1,c}$ can be seen as a ``weighted'' version of the map $\tb_1$ (Definition \ref{def_btilde+}),
in the sense that it shifts an $e$-period one step to the right in the $l$-abacus,
which is determined by $c$ (and is not necessarily the first trivial one).
\item By Remark \ref{rem_ribbon}, a Heisenberg crystal operator acts on a doubly highest weight vertex $|\bla,\bs\rangle$ by adding
an $e$-ribbon to the corresponding charged partition $|\bla,\bs\rangle$.
\item The terminology ``operator'' may seem abusive, since the Heisenberg crystal operators
are simply defined as combinatorial maps, unlike the Kashiwara crystal operators, which are operators on the vector space $L/qL$ 
(where $L$ is a crystal lattice at $q=0$).
It will be shown in \cite{Gerber2016a} that these maps are in fact specialisations at $q=0$ of some linear operators
on the $\C(q)$-vector space $\La^s$, in complete analogy with the Kashiwara crystal operators.
\end{enumerate}
\end{rem}

Each $l$-charge $\bs\in A(s)$ determines a connected component of the $\cH$-crystal.
A source vertex in the $\cH$-crystal is called a \textit{highest weight vertex}
(by analogy with the quantum group case):
it is a charged partition $|\la,s\rangle$ such that $\tb_{-\si}|\la,s\rangle = 0$ for all $\si\in\Pi$.
In other terms, the highest weight vertices in the $\cH$-crystal are the elements of the form
$|\bar{\la},s\rangle$ for some partition $\la$.
The number of arrows necessary to go from $|\bar{\la},s\rangle$ to $|\la,s\rangle$ 
in the $\cH$-crystal is called the \textit{depth} of $|\la,s\rangle$ and is equal to $|\ka(\la)|$.

\medskip

By definition, a map $\tb_{\si}$ (with $\si$ a partition) 
is a composition of maps of the form $\tb_{z}$ with $z$ positive integer.
We can now give an alternative description of $\tb_{\si}$ using composition of
Heisenberg crystal operators.
Let $\{ \ga_k \, ; \, k=1,\dots,|\si|\}$ be the set of boxes of $\si$, ordered from bottom
to top, and from right to left. If $\ga_k=(a_k,b_k)$ (row and column indices),
then write $c_k = b_k-a_k$. 
In particular, one always has $c_{|\si|}=0$.
We have
\begin{equation}\label{heis_dec}
\tb_{\si} = \tb_{1,c_1}\circ \tb_{1,c_2}\circ \dots \circ\tb_{1,c_{|\si|}}.
\end{equation}

\begin{thm}\label{thm_heis_op}
The Heisenberg crystal operators simultaneously commute with the $\Ue$-crystal operators
and with the $\Ul$-crystal operators when computed with respect to Indexation (\ref{indexation}).
\end{thm}

\proof
By definition, the Heisenberg crystal operators are a composition of a map $\tb_{-\ka}$ and a map $\tb_{\si}$.
Both these maps are double crystal-isomorphisms by Propositions \ref{prop_btildekappa} and \ref{prop_btildekappa+}.
This proves the claim.
\endproof

To sum up, we have constructed a new crystal structure, so that we have in total three crystal structures on
the space $\La^s$:
\begin{itemize}
 \item[-] a $\Ue$-crystal,
\item[-] a $\Ul$-crystal,
\item[-] an $\cH$-crystal,
\end{itemize}
which are all explicited and pairwise commute
provided one uses the correspondence (\ref{indexation}) to switch between the different indexations.

\subsection{The decomposition theorem}\label{crys_decomp}\

\begin{num}
Let $r$ (respectively $t$) be a non-negative integer and, 
let $i_1,\dots,i_r$ (respectively $j_1,\dots,j_t$) be elements of $\{0,\dots,e-1\}$ (respectively $\{0,\dots,l-1\}$).
We denote
$$\tF_{(i_1 \dots i_r)} = \tf_{i_1}\dots\tf_{i_r} \mand \tdF_{(j_1 \dots j_t)} = \tdf_{j_1}\dots\tdf_{j_t}.$$ 
\end{num}

The following theorem says that every charged $l$-partition 
is obtained from the empty $l$-partition charged by an element of $A(s)$
by applying some crystal operators of $\Ue$, of $\cH$, and of $\Ul$.
So this is an analogue of Theorem \ref{thm_uglov} at the crystal level.

\begin{thm}\label{thm_crys_decomp}
For all charged $l$-partition $|\bla,\bs\rangle$, there exist $r,t\in\Z_{\geq0}$,
$i_1,\dots,i_r\in\{0,\dots,e-1\}$, $j_1,\dots,j_t\in\{0,\dots,l-1\}$,
and a partition $\si$ such that
$$|\bla,\bs\rangle = (\tdF_{(j_1 \dots j_t)} \circ \tb_{\si} \circ \tF_{(i_1 \dots i_r)}) \, \, |\bemp,\tbs\rangle,
$$
for some $\tbs\in A(s)$.
\end{thm}

Here, we have implicitely used the correspondence (\ref{indexation})
to switch between the indexations by $l$-partitions, partitions, and $e$-partitions.
Note that by Identity (\ref{heis_dec}), the map $\tb_{\si}$ in the middle is indeed a composition of
Heisenberg crystal operators.

\proof We identify as usual the $l$-partitions, $e$-partitions, and $1$-partitions using (\ref{indexation}).
Starting from $|\la,s\rangle$, one first goes
back in the $\Ul$-crystal to the highest weight vertex, say $|\nu,s\rangle$.
One then computes $\tb_{-\ka(\nu)}|\nu,s\rangle=|\bar{\nu},s\rangle$.
Finally, one can go back in the $\Ue$-crystal to the highest weight vertex.
By Theorem \ref{thm_heis_op}, the order of these operations does not matter,
and by Proposition \ref{prop_btildekappa_1}, the resulting $l$-partition is empty,
and charged by an element of $A(s)$ according to Proposition \ref{prop_charge}.
\endproof

\subsection{An application using FLOTW multipartitions}\label{flotw}\

A consequence of Theorem \ref{thm_crys_decomp} is the existence of a labelling
of each charged $l$-partition by a triple consisting of 
a particular $l$-partition, a partition and a particular $e$-partition.
More precisely, let us introduce the convenient class of \textit{FLOTW} multipartitions.

\begin{defi} \label{def_flotw}
Let $\bla=(\la^1,\dots,\la^l)$ be an $l$-partition and $\bs=(s_1,\dots,s_l)$ be an $l$-charge in $A(s)$ (cf Notation \ref{fund_dom}).
For $j\in\{1,\dots,l\}$, write $(\la^j = (\la_1^j,\la_2^j,\dots)$.
We call $|\bla,\bs\rangle$ \textit{FLOTW} if the two following conditions are satisfied.
\begin{enumerate}\item \begin{itemize}
  \item $\la_k^j\geq \la_{k+s_{j+1}-s_j}^{j+1} \,\, \forall j\in\{ 1,\dots, l-1\}$ and $\forall k\geq 1$, and
\item $\la_k^l\geq \la_{k+e+s_1-s_l}^1\,,\forall k\geq 1$.
\end{itemize}
 \item For all $\al>0$, the residues of the rightmost boxes of the parts of size $\al$ do not cover $\{0,\dots, e-1 \}$.
\end{enumerate}
\end{defi}
Denote by $\Psi_{\bs}$ the set of FLOTW $l$-partitions with charge $\bs$,
and set $$\Psi=\{ |\bla,\bs\rangle \in\Psi_\bs \; |\; \bs\in A(s)\}.$$

We define similarly the level-rank duals $\dPsi_{\dbs}$ and $\dPsi$ by exchanging the roles of $e$ and $l$.

\begin{rem}
Throughout this paper, we have assumed that $l>1$.
This definition is however still valid when $l=1$. In this case, $l$-partitions are simply partitions,
and the FLOTW partitions are precisely the $e$-regular partitions (and in this case, the charge is insignificant).
\end{rem}

\begin{thm}[\mbox{\cite[Theorem 2.10]{FLOTW1999}}]\label{thm_flotw}
The vertices of the connected component of 
the $\Ue$-crystal graph of $\cF_{\bs,e}$ containing $|\bemp,\bs\rangle$
are the FLOTW $l$-partitions.
\end{thm}

The relevance of this theorem is that 
a priori, the vertices in the crystal graph of $\cF_{\bs,e}$ are computable, but only have a recursive definition:
one starts with the highest weight vertex and recursively applies some crystal operators of the form $\tf_i$;
whereas the FLOTW $l$-partitions have a more explicit (in particular non-recursive) combinatorial definition.

\begin{exa}
Take $e=4$, $l=2$ and $\bs=(0,1)$. Then the elements of $\Psi_{\bs}$ of rank $4$ are 
$$ {\tiny
\begin{array}{llll}
(\; \emptyset \; , \; \young(123) \;)
&
(\; \young(0) \; , \; \young(1,0) \;)
&
(\; \young(0) \; , \; \young(12) \;)
&
(\; \young(0,\moinsun) \; , \; \young(1) \;)
\\
(\; \young(0,\moinsun,\moinsdeux) \; , \; \emptyset \;)
&
(\; \young(01) \; , \; \young(1) \;)
&
(\; \young(01,\moinsun) \; , \; \emptyset \;)
&
(\; \young(012) \; , \; \emptyset \;)
.
\end{array}
}
$$
\end{exa}

\medskip

Recall that we had $B=\{ |\bla,\bs\rangle \; ; \; \bla\in\Pi_l \}$.

\begin{cor}\label{cor_param}
There is an injection
$$
\begin{array}{ccc}
B & \overset{\iota}{\longrightarrow}  & \Psi \times \Pi \times \dPsi. 
\end{array}
$$
\end{cor}

\proof
By Theorem \ref{thm_crys_decomp}, one can always decompose any element of $B$ as follows
$$|\bla,\bs\rangle = (\tdF_{(j_1 \dots j_t)} \circ \tb_{\si} \circ \tF_{(i_1 \dots i_r)}) \, \, |\bemp,\tbs\rangle.
$$
Now, by Theorem \ref{thm_flotw}, the $l$-partition $\tF_{(i_1 \dots i_r)}  |\bemp,\tbs\rangle$ is FLOTW,
i.e. an element of $\Psi$. Denote it $|\bmu,\br\rangle$.
Similarly $\tdF_{(j_1 \dots j_r)}  |\dbemp,\dtbs\rangle \in \dPsi$, and we denote it
$|\dbnu,\dbt\rangle$.
Therefore, we get a map 
$$
\begin{array}{cccc}
\iota: & B & \lra  & \Psi \times \Pi \times \dPsi \\
 & |\bla,\bs\rangle & \xleftrightarrow[]{\hspace{1.2cm}} & (|\bmu,\br\rangle,\si,|\dbnu,\dbt\rangle).
\end{array}
$$
To show that $\iota$ is injective, let $|\bla^{(1)},\bs\rangle$ and $|\bla^{(2)},\bs\rangle$ be two
elements of $B$ such that $\iota(|\bla^{(1)},\bs\rangle)=\iota(|\bla^{(2)},\bs\rangle)$.
For $k=1,2$, write $$|\bla^{(k)},\bs\rangle = (\tdF_{(j_1^{(k)} \dots j_{t_k}^{(k)})} \circ \tb_{\si^{(k)}} \circ \tF_{(i_1^{(k)} \dots i_{r_k}^{(k)})}) \, \, |\bemp,\tbs^{(k)}\rangle$$
using Theorem \ref{thm_flotw}.
By definition of $\iota$ and since $\iota(|\bla^{(1)},\bs\rangle)=\iota(|\bla^{(2)},\bs\rangle)$, we have
\begin{itemize}
\item $\tF_{(i_1^{(1)} \dots i_{r_1}^{(1)})}  |\bemp,\tbs^{(1)}\rangle=\tF_{(i_1^{(2)} \dots i_{r_2}^{(2)})}  |\bemp,\tbs^{(2)}\rangle$
\item $\tdF_{(j_1^{(1)} \dots j_{t_1}^{(1)})} |\dbemp,\dtbs^{(1)}\rangle = \tdF_{(j_1^{(2)} \dots j_{t_2}^{(2)})} |\dbemp,\dtbs^{(2)}\rangle$
\item $\tb_{\si^{(1)}} |\la^{(1)},\ts\rangle = \tb_{\si^{(2)}} |\la^{(2)},\ts\rangle$.
\end{itemize}
So in particular, $r_1=r_2=:r$, $t_1=t_2=:t$, $\tbs^{(1)}=\tbs^{(2)}=:\tbs$ and $\dtbs^{(1)}=\dtbs^{(2)}=:\dtbs$.
Using the commutation properties of Theorems \ref{thm_comm_crystal} and \ref{thm_heis_op}, we have
$$
\begin{array}{rcl}
|\bla^{(1)},\bs\rangle & = & \tdF_{(j_1^{(1)} \dots j_{t}^{(1)})} \circ \tb_{\si^{(1)}} \circ \tF_{(i_1^{(1)} \dots i_{r}^{(1)})} \, \, |\bemp,\tbs\rangle \\
& = & \tdF_{(j_1^{(1)} \dots j_{t}^{(1)})} \circ \tb_{\si^{(1)}} \circ \tF_{(i_1^{(2)} \dots i_{r}^{(2)})} \, \, |\bemp,\tbs\rangle \\
& = & \tF_{(i_1^{(2)} \dots i_{r}^{(2)})} \circ \tdF_{(j_1^{(1)} \dots j_{t}^{(1)})} \circ \tb_{\si^{(1)}}   \, \, |\emptyset,\ts\rangle \\
& = & \tF_{(i_1^{(2)} \dots i_{r}^{(2)})} \circ \tdF_{(j_1^{(1)} \dots j_{t}^{(1)})} \circ \tb_{\si^{(2)}}   \, \, |\emptyset,\ts\rangle \\
& = &  \tF_{(i_1^{(2)} \dots i_{r}^{(2)})} \circ \tb_{\si^{(2)}} \circ \tdF_{(j_1^{(1)} \dots j_{t}^{(1)})}   \, \, |\dbemp,\dtbs\rangle \\
& = &  \tF_{(i_1^{(2)} \dots i_{r}^{(2)})} \circ \tb_{\si^{(2)}} \circ \tdF_{(j_1^{(2)} \dots j_{t}^{(2)})}   \, \, |\dbemp,\dtbs\rangle \\
& = & |\bla^{(2)},\bs\rangle.
\end{array}
$$
\endproof

\renewcommand{\Im}{\operatorname{Im}}

\begin{rem}\label{rem_param}
The map $\iota$ is not surjective, since for arbitrary
$|\bmu,\br\rangle\in\Psi$ and $|\dbnu,\dbt\rangle\in \dPsi$,
one does not necessary have $|\bemp,\br\rangle \longleftrightarrow |\dbemp, \dbt\rangle$ in Correspondence \ref{indexation}.
In fact, the image of $\iota$ is of the form $\Im(\iota)=\Phi\times\Pi\times\dot{\Phi}$,
where  ${\Phi}\subset\Psi$ and $\dot{\Phi}\subset\dPsi$.
In \cite[Remark 6.5]{Gerber2015}, we have constructed an affine analogue of the Robinson-Schensted correspondence,
which maps bijectively an element $|\bla,\bs\rangle\in B$ to a pair consisting of an FLOTW $l$-partition 
$|\bmu,\br\rangle$ and a combinatorial ``recording data'' $(\underline{\sQ},\underline{\al})$.
The FLOTW $l$-partition is precisely the one appearing in the above theorem.
It would be interesting to determine $\Phi$ and $\dot{\Phi}$ and to investigate the relationship
between $\Phi\times\Pi\times\dot{\Phi}$ and the set of pairs  $(|\bmu,\br\rangle , (\underline{\sQ},\underline{\al}))$ appearing in \cite{Gerber2015}.
\end{rem}

\section{Application to the representation theory of cyclotomic rational Cherednik algebras}\label{cher}

For a charge $\bs\in\Z^l(s)$ and a non-negative integer $n$, one associates the Cherednik algebra $H_{\bc,n}$
with parameter $\bc = (-\frac{1}{e},\bs)$
arising from the complex reflection group $G(l,1,n)= (\Z/l\Z)^n \rtimes \fS_n$
(this is the so-called cyclotomic case).
The parameter $\bc$ is sometimes expressed differently in the literature.
For some background, one can refer to e.g. \cite{Shan2011}.

There is a corresponding category $\cO$, see \cite{GGOR2003} for its definition, 
denoted $\cO_{\bc,n}$, and one can consider, for $n$ varying, all categories $\cO_{\bc,n}$
together. Denote it $\cO_\bc$.
The simple objects in $\cO_\bc$ are parametrised by the elements of $\Irr( G(l,1,n) )$ for $n$ varying,
i.e. by $l$-partitions.

It is known that the Fock space plays an important role in the representation theory of 
$H_{\bc,n}$, with $n$ varying, via categorification phenomenons.
In particular, the crystal of $\cF_{\bs,e}$ is categorified by the branching rule on $H_{\bc,n}$ with $n$ varying,
where the Kashiwara operators $\te_i$ (respectively $\tf_i$) correspond to the parabolic restriction 
(respectively induction) in $\cO_\bc$, see Shan \cite[Theorem 6.3]{Shan2011} and Losev \cite[Theorem 5.1]{Losev2013}.

Moreover, the action of the Heisenberg algebra (cf Section \ref{sec_uglov})
has also been categorified by Shan and Vasserot \cite{ShanVasserot2012},
and some of the associated combinatorics has been recently studied by Losev \cite{Losev2015}.

\subsection{Interpretation of the crystal level-rank duality}\label{cher_level_rank}\

Recall that the crystal level-rank duality is given by Correspondence (\ref{indexation}),
thanks to which one can compute the two commuting quantum group crystals, see Section \ref{comm_crystals}.
This is a combination of Uglov's level-rank duality given by Correspondence (\ref{indexation_ug})
and of the conjugation isomorphism.

The categorical interpretations of these dualities have been studied by Shan, Varagnolo and Vasserot \cite{ShanVaragnoloVasserot2014a}, 
by Rouquier, Shan, Varagnolo and Vasserot \cite{RSVV2016}
and by Webster \cite{Webster2013a} to prove conjectures of Rouquier \cite{Rouquier2008} and \cite{ChuangMiyachi2011}.
In fact, it is known that Uglov's level-rank duality between $\cF_{\bs,l}$ and $\cF_{\dbs,e}$ is categorified by
the Koszul duality between the corresponding Cherednik categories $\cO$, see for instance \cite[Section 6]{Webster2013a}.
Moreover, conjugating is categorified by the Ringel duality \cite[Section 6.2.2]{RSVV2016}.
The composition of both, that is Correspondence (\ref{indexation}) is therefore categorified by the composition of the Ringel and Koszul dualities,
sending simple to tilting modules in the respective categories $\cO$ and
giving rise to an equivalence of bounded derived categories \cite[Conjecture 6]{ChuangMiyachi2011}, \cite[Theorem 7.4]{RSVV2016}.
It commutes with the categorical $\sle$-crystal and $\sll$-crystal arising from Bezrukavnikov-Etingof's parabolic induction functors in the corresponding categories $\cO$ \cite{Shan2011}.

\subsection{Propagation in the $\Ul$-crystal and compatibility with the results of Losev}\label{cher_losev}\

In \cite{Losev2015}, Losev has introduced a combinatorial recipe to compute a so-called $\fs\fl_\infty$-crystal
\footnote{This terminology is justified by the same kind of arguments as that of Remark \ref{rem_heis_crys}.}
on the set of charged $l$-partitions which reflects, at a combinatorial level, an abstract crystal structure on the set of classes of simple objects
in the category $\cO_e$, arising from the action of the Heisenberg algebra at a categorical level
(whose existence goes back to Shan and Vasserot \cite{ShanVasserot2012}).

This recipe consists of two ingredients:
\begin{itemize}
 \item An explicit description of some operators $\ta_\si$ 
 (parametrised by partitions $\si$, first introduced in \cite{ShanVasserot2012}) 
 on charged $l$-partitions, in the case where the $l$-charge is \textit{asymptotic}.
 \item A formula for \textit{wall-crossing bijections}, that permits to pass from the asymptotic case to the general case.
\end{itemize}
Notice that the formula for these wall-crossing bijections is unfortunately not very explicit.
Moreover, these ingredients are introduced for highest weight vertices in the $\Ue$-crystal;
however, the commutation of this $\fs\fl_\infty$-crystal with the $\Ue$-crystal
ensures that one can extend it to the whole set of partitions (see \cite[Remark 5.4]{Losev2015}).
Finally, Losev does not use the combinatorial level-rank duality at all, and there is no triple crystal structure involved.
In this section, we will show that Losev's $\fs\fl_\infty$-crystal coincides with the Heisenberg crystal introduced in Section \ref{btilde_inv} above.

\medskip

Let $j_0\in\{1,\dots,l\}$. An $l$-charge $\bs$ is called \textit{$j_0$-asymptotic}
if there exists a positive integer $N$ such that $s_{j_0} > s_j+N$ for all $j\in\{1,\dots,l-1\}, j\neq j_0$.
Actually, in what follows, we will consider the maximal such $N$ for simplicity.
In this case, we will also call an element $|\bla,\bs\rangle$ \textit{asymptotic (charged) multipartition}.

\begin{lem}\label{lemosev}
Let $\bs$ be an $j_0$-asymptotic $l$-charge.
If $|\bla,\bs\rangle$ is a highest weight vertex in the $\Ue$-crystal such that $|\bla|\leq N$, 
then there exists a partition $\theta=(\theta_1,\theta_2,\dots)$ such that
$\la^{j_0} = (\theta_1^e,\theta_2^e,\dots)$.
\end{lem}

\proof
Because of Theorem \ref{thm_hw_vertices}, we know that $\cA(\bla,\bs)$ is totally $e$-periodic.
In view of the condition on the rank of $\bla$, the first periods of $\cA(\bla,\bs)$ consist
only of elements of the form $(j,\be)$ with $j=j_0$ (in other terms, the first
periods are entirely included in the $j_0$-th row of the abacus).
Hence, the partition $\la^{j_0}$ is of the form $(\theta_1^e,\theta_2^e,\dots)$ for some non-negative integers $\theta_i$.
\endproof

Notice that this partition $\theta$ is constructed in a similar way as the partition $\ka$ for doubly highest weight vertices
(except that for $\theta$, one focuses exclusively on the $j_0$-th component of $|\bla,\bs\rangle$).
We will show in Theorem \ref{thmosev2} that $\theta$ is in fact the partition $\ka$ associated to the corresponding 
doubly highest weight vertex.

Let us now recall the result of Losev that is relevant in our context.
Let $\bs$ be an $j_0$-asymptotic $l$-charge, and
$|\bla,\bs\rangle$ be a highest weight vertex in the $\Ue$-crystal such that $|\bla|+e|\theta|\leq N$
(cf Lemma \ref{lemosev} above).
The following is \cite[Section 5.1.2 and Proposition 5.3]{Losev2015}.

\begin{thm}\label{thmosev}\
\begin{enumerate}
 \item The depth of $|\bla,\bs\rangle$ in the $\fs\fl_\infty$-crystal is equal to the rank of $\theta$.
\item If $\theta=\emptyset$ and $\si=(\si_1,\si_2,\dots)$ is a partition such that $|\bla|+e|\si|\leq N$, then
$\ta_\si (|\bla,\bs\rangle)=|\bmu,\bs\rangle$, where
$\mu^j=\la^j$ for all $j\neq j_0$, and $\mu^{j_0}=(\si_1^e,\si_2^e,\dots)$.
\item The $\fs\fl_\infty$-crystal commutes with the $\Ue$-crystal.
\end{enumerate}
\end{thm}

Here, we have slightly ``rephrased'' the original result of Losev.
In particular, the notion of asymptoticity must be reversed in order to be compatible with the language of Fock spaces,
as well as the convention on ``multiplicating/dividing'' of partitions by $e$ 
(one recovers Losev's convention by conjugating, see also \cite[Section 5.5]{Losev2015}).

\begin{exa}
Take $l=2$, $\bla = (1^3,\emptyset)$ , $\bs= (0,14)$, $e=3$ and $\si=(2,1)$.
So we have $j_0=2$ and $N=s_2-s_1-1=13$. Since $|\bla|+e|\si|=3+3.3=12$,
we are in the conditions of the theorem.
The abacus of $|\bla,\bs\rangle$ is
\begin{figure}[H] 
\includegraphics{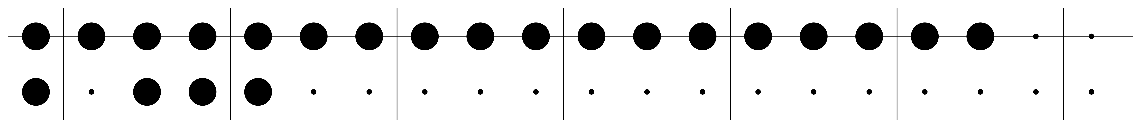}
\label{ab22}
\end{figure}
Applying $\ta_\si$, we get the following abacus
\begin{figure}[H] 
\includegraphics{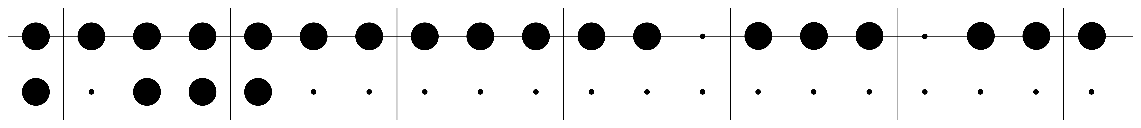}
\label{ab23}
\end{figure}
So we see that $\ta_\si$ acts by shifting periods in $\la^{j_0}$ to the right according to $\si$:
the first period is shifted two steps and the second one step.
\end{exa}

One first thing to notice is that Losev's formula for $\ta_\si$ is 
similar to the formula of the operators $\tb_\si$ of Section \ref{heiscrys}
(shifting $e$-periods to the right).
However, one sees that the property of being asymptotic is somehow antagonistic to the property of being
a doubly highest weight vertex. More precisely, a doubly highest weight vertex can never be asymptotic;
and conversely, an asymptotic multipartition can never be a doubly highest weight vertex (except for the trivial cases).
This is clear for instance looking at Proposition \ref{prop_charge}.
Still, we have explained how to extend the definition of the new operators to the whole set of partitions, in (\ref{inducedbtilde}).
In this section, we will show that the operator $\ta_\si$ 
actually coincides with $\tb_\si$ for all partition $\si$.
Moreover, we show that the partition $\theta$ arising in the asymptotic case is in fact
equal to the partition $\ka$ arising in the doubly highest weight vertex case.
Note that the maps $\tb_{-1}$ and $\tb_{-\ka}$ implicitely then corresponds to 
taking the inverse maps to $\ta_{(1)}$ and $\ta_{\ka}$.

\medskip

For every asymptotic charged $l$-partition $|\bla,\bs\rangle$ which is a highest weight vertex in the $\Ue$-crystal, 
one can consider the corresponding doubly highest weight vertex. 
One can apply to it an operator $\tb_{-\ka}$, and go back to the highest weight vetex in the $\Ul$-crystal 
to get the corresponding $l$-partition $|\bmu,\br\rangle=\tb_{-\ka}|\bla,\bs\rangle$
(cf Procedure (\ref{inducedbtilde})).
In fact, the propagation in the $\Ul$-crystal turns out to have a nice description:
acting by $\tb_{-\ka}$  and by $\tb_\si$ on $|\bla,\bs\rangle$ is
combinatorially ``the same'' as acting on doubly highest weight vertices (i.e. shifting $e$-periods to the left), 
as is stated in the next proposition.

\begin{prop}\label{prop_losev_1}
Let $|\bla,\bs\rangle$ be a highest weight vertex in the $\Ue$-crystal.
Write $|{\tbla},\tbs\rangle$ for
the corresponding doubly highest weight vertex, and set $\ka=\ka(|{\tbla},\tbs\rangle)=(\ka_1,\ka_2,\dots,\ka_{N})$.
Then $|\bla,\bs\rangle$ has at least $N$ non-trivial periods, and
\begin{enumerate}
\item $\tb_{-\ka}$ acts on $|\bla,\bs\rangle$ by shifting the $k$-th period 
of $\cA(\bla,\bs)$ by $\ka_k$ steps to the left, for all $k=1,\dots,N$, 
starting from $k=N, \dots$ and finishing by $k=1$,
\item if $\ka=\emptyset$, then for all partition $\si=(\si_1,\si_2,\dots,\si_r)$, $\tb_{\si}$ acts on $|\bla,\bs\rangle$ 
by shifting the $k$-th period of $\cA(\bla,\bs)$ by $\si_k$ steps to the right, for all $k=1,\dots,r$, 
starting from $k=r, \dots$ and finishing by $k=1$.
\end{enumerate}
\end{prop}

\begin{rem}
Note that in the case $|{\bla},\bs\rangle=|{\tbla},\tbs\rangle$, this procedure
coincides with the procedure for $\tb_{-\ka}$ described in Definition \ref{def_btildepart}.
The subtlety here is that we might have more than $N$ non-trivial periods in $|\bla,\bs\rangle$,
so we have to modify the statement.
\end{rem}

\proof
We have already explained in Section \ref{comm_crystals} how the crystal operators $\tdf_j$ of $\Ul$ act on the $l$-abacus,
see Section \ref{comm_crystals}.
They are $\Ue$-crystal isomorphisms, and in fact they transform an $e$-period $P=((j_k,\be_k))_{k=1,\dots,e}$
of $\cA(\tbla,\tbs)$ into another $e$-period $P'=((j'_k,\be'_k))_{k=1,\dots,e}$,
where either 
\begin{itemize}
 \item $(j'_k,\be'_k)=(j_k,\be_k)$ for all $k$ but one, denoted $k_0$, for which $(j'_{k_0},\be'_{k_0})=(j_{k_0}-1,\be_{k_0})$, or
 \item $(j'_k,\be'_k)=(j_k,\be_k)$ for all $k>1$ and $(j'_1,\be'_1)=(l,\be_e+e)$
 (in which case $j_e=1$).
\end{itemize}
This is true because $\cA(\tbla,\tbs)$ is totally $e$-periodic
(and so are all elements in the $\Ul$-crystal).
We see directly that such a procedure can only preserve the number of periods, or make it increase,
proving the opening statement.

Let us prove Point (1).
We first need to show that it is possible to apply the shifting procedure
in the proposition,
that is to say, that each element of the $N$-th period has at least 
$\ka_{N}$ empty spots to its left, and so on
(formally, $(j,\be)\in P_{N} \Ra (j,\be-a)\notin \cA(\bla,\bs) \, \forall a=1,\dots,\ka_{N}$,
and so on, if $P_k$ denotes the $k$-th period in $\cA(\bla,\bs)$).
By contradiction, suppose that
applying the operator $\tdf_j$ to a highest weight vertex in the $\Ue$-crystal 
moves a bead of a period $P_k$ in the abacus to a spot $(j,\be)$ ($j\in\{1,\dots,l\}$ and meaning $\tdf_0$ if $j=l$) 
such that a period $P_{k'}$, with $k>k'$, contains an element of the form $(j,\be')$ and has exactly $\ka_{k'}$
empty spots to its left.
In this case, the element $(j,\be')$ creates a $-$ in the $j$-word $\dw_j(|\dbla',\dbs'\rangle)$, which directly
simplifies with the $+$ created by the bead that is moved by $\tdf_j$,
which is a contradiction.

In fact, this procedure indeed gives the crystal action of the Heisenberg algebra for
the highest weight vertices in the $\Ue$-crystal.
It suffices to notice that shifting the considered $e$-periods preserves the reduced $j$-words $\dw_j$,
because this amounts to potentially make subwords of the form $(-+)$ collapse.
In addition, one observes that $\tdf_j$ acts on the modified $l$-abacus by 
moving the bead corresponding to the bead of the original abacus $\cA(|\bla,\bs\rangle)$
which is moved by $\tdf_j$.
This is the same as applying the procedure to $\tdf_j|\bla,\bs\rangle$.
Because $\tb_{-\ka}$ is a $\Ul$-crystal isomorphism (Proposition \ref{prop_btildekappa}),
this procedure is indeed the action of $\tb_{-\ka}$ on $l$-abaci.

Using the exact same arguments and looking at the reverse procedure,
Point (2) is also proved.
\endproof

\begin{thm}\label{thmosev2}
The Heisenberg crystal coincides with Losev's $\fs\fl_\infty$-crystal.
\end{thm}

\proof
It suffices to show that $\tb_\si = \ta_\si$ for all $\si\in\Pi$, 
and that $\theta=\ka$.

In fact, we first show that $\tb_\si$
and $\ta_\si$ coincide on highest weight vertices in the $\Ue$-crystal
which are $j_0$-asymptotic for some $j_0\in\{1,\dots,l\}$.
This is enough because we know that in both cases, the maps commute
with the crystal operators of $\Ue$ 
(Theorem \ref{thm_heis_op} for $\tb_\si$ and Theorem \ref{thmosev} for $\ta_\si$).
Every such charged $l$-partition $|\bla,\bs\rangle$ 
is obtained from a doubly highest weight vertex
by applying a sequence of Kashiwara crystal operators $\tdf_{j_1}\tdf_{j_2}\dots\tdf_{j_r}$.
By Proposition \ref{prop_losev_1}, we know how these operators act.
In the asymptotic case, if $|\bla|+e|\si|\leq N$ 
(where the $N$ comes from the asymptotic property),
applying $\tb_{\si}$ only affects the $j_0$-th row of $\cA(\bla,\bs)$.
Moreover, the shifting procedure on abaci described in Proposition \ref{prop_losev_1}
is exactly Losev's formula for $\ta_{\si}$ on charged $l$-partitions, see Theorem \ref{thmosev}.
So we have $\tb_\si = \ta_\si$ for all partition $\si$.

Similarly, the action of $\tb_{-\ka}$ is entirely described on the
$j_0$-th row of $\cA(\bla,\bs)$,
and the procedure of Proposition \ref{prop_losev_1} on abaci
is in this case precisely the reverse procedure 
of Losev's formula for $\ta_{\theta}$ on $l$-partitions, with $\theta=\ka$.
Therefore, $\ka=\theta$.
In particular, the depth of $|\bla,\bs\rangle$ in the Heisenberg crystal is by definition 
$|\ka|=|\theta|$.
\endproof

Therefore, we can now use the results of Losev \cite{Losev2015} and Shan and Vasserot \cite{ShanVasserot2012}
on the Heisenberg crystal.

\subsection{A combinatorial characterisation of finite-dimensional simple modules}\label{prim}\

One important result of Shan and Vasserot is \cite[Proposition 5.18]{ShanVasserot2012},
which gives a characterisation of the finite-dimensional simple modules for cyclotomic rational Cherednik algebras.
They show that this property is equivalent to being ``primitive''.
Combinatorially, this amounts to saying that the $l$-partition 
labelling this module is simultaneously a highest weight vertex in the $\Ue$-crystal and
in the Heisenberg crystal, see e.g. \cite[Section 5.1.1]{Losev2015}.

Using the results of Section \ref{crys_decomp},
we can give an explicit combinatorial description of these $l$-partitions.
For this, recall that we have introduced the notion of FLOTW $e$-partitions in Definition \ref{def_flotw},
and that we can use the correspondence (\ref{indexation}) between $l$-partitions charged by $\bs$ and
$e$-partitions charged by $\dbs'$.

\begin{thm}\label{thm_cher}
A simple $H_{\bc,n}$-module is finite-dimensional 
if and only if it is labelled by an $l$-partition $\bla$ of rank $n$ such that
$|\dbla',\dbs'\rangle$ is an FLOTW $e$-partition.
\end{thm}

\proof
As already explained, a simple $H_{\bc,n}$-module is finite-dimensional 
if and only if it is labelled by an $l$-partition $\bla$ of rank $n$ such that
$|\bla,\bs\rangle$ is a highest weight vertex in the $\Ue$-crystal and
$|\la,s\rangle$ is a highest weight vertex in the Heisenberg crystal, 
i.e. $\la=\bar{\la}$.

Assume first that 
$|\bla,\bs\rangle$ is a highest weight vertex in the $\Ue$-crystal and
$|\bla,\bs\rangle$ is a highest weight vertex in the Heisenberg crystal.
Then with the notation of Theorem \ref{thm_crys_decomp}, we have 
$(i_1,\dots,i_r)=\emptyset$ and $\si=\emptyset$, thus
$$|\dbla',\dbs'\rangle = \tdF_{(j_1,\dots,j_k)}|\dbemp,\dtbs\rangle$$
where $\dtbs\in \dA(s)$.
By Theorem \ref{thm_flotw}, $|\dbla',\dbs'\rangle \in\dPsi$, i.e. is FLOTW.

Conversely, if $|\dbla',\dbs'\rangle$ is an FLOTW $e$-partition, then there exist some
$(j_1,\dots,j_k) \in \{0,\dots,l-1\}^k$ such that 
$|\dbla',\dbs'\rangle = \tdF_{(j_1,\dots,j_k)}|\dbemp,\dtbs\rangle$,
for some $\dtbs\in \dA(s)$.
Then for all $\si\in\Pi$, we have
$$
\begin{array}{ccll}
\tb_{-\si} |\dbla',\dbs'\rangle &  = & (\tb_{-\si} \circ \tdF_{(j_1,\dots,j_k)}) |\dbemp,\dtbs\rangle & \\
& = & (\tdF_{(j_1,\dots,j_k)} \circ\tb_{-\si})|\dbemp,\dtbs\rangle & \text{by Theorem \ref{thm_heis_op}} \\
& = & 0. & \end{array}
$$
and for all $i\in\{0,\dots,e-1\}$,
$$
\begin{array}{ccll}
\te_i |\dbla',\dbs'\rangle &  = & (\te_i \circ \tdF_{(j_1,\dots,j_k)}) |\dbemp,\dtbs\rangle & \\
& = & (\tdF_{(j_1,\dots,j_k)} \circ\te_i)|\dbemp,\dtbs\rangle & \text{by Theorem \ref{thm_comm_crystal}} \\
& = & 0. & \end{array}
$$
So $|\bla,\bs\rangle$ is a highest weight vertex in the $\Ue$-crystal and
$|\la,s\rangle$ is a highest weight vertex in the Heisenberg crystal.
\endproof

\begin{rem}\label{rem_cusp}
In the context of modular representations of finite classical groups, there is a
characterisation of cuspidal unipotent modules by Dudas, Varagnolo and Vasserot \cite[Theorem 5.11]{DudasVaragnoloVasserot2016} similar to Shan and Vasserot's
characterisation of finite-dimensional modules for cyclotomic Cherednik algebras.
Therefore, we obtain an explicit description of these cuspidal modules which is exactly that of Theorem \ref{thm_cher} with $l=2$,
see also \cite[Section 5.5.3]{DudasVaragnoloVasserot2016}.
\end{rem}

\textbf{Ackowledgments:} 
I thank Olivier Dudas, Nicolas Jacon, C\'edric Lecouvey
and Peng Shan for many useful conversations.
I would also like to thank the organisers of the
conference \textit{Categorical Representation Theory and Combinatorics} held at KIAS, Seoul,
during which parts of this paper were completed.
Extra thanks to Philippe Nadeau, Emily Norton and Galyna Dobrovolska and to the anonymous referee for suggesting some 
improvements to this article.

\bibliographystyle{plain}

\end{document}